\documentclass[11pt]{article}

\usepackage{graphicx}

\usepackage{bbm}

\usepackage{amsmath}
\usepackage{amsthm}
\usepackage{amsfonts}
\usepackage{amssymb}
\usepackage{mathtools}
\usepackage{xcolor}
\RequirePackage[colorlinks,citecolor=blue,urlcolor=blue]{hyperref}

\newcommand{\eee}{{\rm e}}

\newcommand{\me}{\mathbb{E}}

\newcommand{\mn}{\mathbb{N}}
\newcommand{\mr}{\mathbb{R}}

\DeclareMathOperator{\1}{\mathbbm{1}}

\newcommand{\ai}[1]{\textcolor{orange!80!black}{#1}}

\newcommand{\mmp}{\mathbb{P}}

\newtheorem{thm}{Theorem}[section]
\newtheorem{lemma}[thm]{Lemma}

\newtheorem{assertion}[thm]{Proposition}
\theoremstyle{definition}

\theoremstyle{remark}
\newtheorem{rem}[thm]{Remark}

\begin{document}
\title{First passage times for decoupled random walks}\date{}
\author{Alexander Iksanov\footnote{Faculty of Computer Science and Cybernetics, Taras Shevchenko National University of Kyiv, Ukraine; e-mail address:
iksan@univ.kiev.ua} \ \ Zakhar Kabluchko\footnote{Department of Mathematics and Computer Science, University of M\"{u}nster, Germany; e-mail address: zakhar.kabluchko@uni-muenster.de} \ \ and \ \ Vitali Wachtel\footnote{Faculty of Mathematics, University of Bielefeld, Germany; e-mail address: wachtel@math.uni-bielefeld.de}}
\maketitle

\begin{abstract}
Motivated by a connection to the infinite Ginibre point process, decoupled random walks were introduced in a recent article Alsmeyer, Iksanov and Kabluchko (2025). The decoupled random walk is a sequence of independent random variables, in which the $n$th variable has the same distribution as the position at time $n$ of a standard random walk with nonnegative increments. We prove distributional convergence in the Skorokhod space equipped with the $J_1$-topology of the  running maxima and the first passage times of decoupled random walks. We show that there exist five different regimes, in which distinct limit theorems arise. Rather different functional limit theorems for the number of visits of decoupled standard random walk to the interval $[0,t]$ as $t\to\infty$ were earlier obtained in the aforementioned paper Alsmeyer, Iksanov and Kabluchko (2025). While the limit processes for the first passage times are inverse extremal-like processes, the limit processes for the number of visits are stationary Gaussian.
\end{abstract}


\hphantom{2020 Mathematics Subject Classification: } Secondary: 60F10

\section{Introduction and main results}

Let $\xi_1$, $\xi_2,\ldots$ be independent copies of a nonnegative random variable $\xi$ with a nondegenerate distribution. Put $S_n=\xi_1+\ldots+\xi_n$ for $n\in\mn$. The random sequence
$(S_n)_{n\geq 1}$ is called {\it random walk} with nonnegative jumps. Let $\hat S_1$, $\hat S_2,\ldots$ be independent random variables such that, for each $n\in\mn$, $\hat S_n$ has the same distribution as $S_n$. Following \cite{Alsmeyer+Iksanov+Kabluchko:2025}, we call the sequence $(\hat S_n)_{n\geq 1}$ {\it decoupled random walk}.

For $t\geq 0$, put $$\hat N(t):=\sum_{n\geq 1}\1_{\{\hat S_n\leq t\}}\quad \text{and}\quad \hat \tau(t):=\inf\{n\geq 1: \hat S_n>t\}.$$ Observe that, for each $t\geq 0$, both $\hat N(t)$ and $\hat \tau(t)$ are almost surely  (a.s.) finite. The equality
$\mmp\{\hat N(t)<\infty\}=1$ is a consequence of $\sum_{n\geq 1}\mmp\{\hat S_n\leq t\}=\sum_{n\geq 1}\mmp\{S_n\leq t\}<\infty$. Furthermore, the inequality $\hat \tau(t)-1\leq \hat N(t)$ a.s. implies that $\hat\tau(t)$ is a.s.\ finite, too. 
The function $t\mapsto \sum_{n\geq 1}\mmp\{S_n\leq t\}$ on $[0,\infty)$ is called {\it renewal function}. It is a standard result of Renewal Theory that the renewal function is finite for all nonnegative arguments under the sole assumption $\mmp\{\xi=0\}<1$. We call the random processes $\hat N:=(\hat N(t))_{t\geq 0}$ and $(\hat \tau(t))_{t\geq 0}$ {\it decoupled renewal process} and {\it decoupled first passage time process}, respectively. We note in passing 
that for the classical walk $(S_n)_{n\geq 1}$ 
with nonnegative increments one has
$\inf\{n\geq 1: S_n>t\}-1 
=\sum_{n\geq 1}\1_{\{S_n\leq t\}}$. 

Some decoupled random walks arise naturally in an investigation of certain determinantal point processes. This connection has served as a motivation behind introducing the decoupled random walks in \cite{Alsmeyer+Iksanov+Kabluchko:2025}. Now we discuss this connection in more details. Let $\mathbb{C}$ denote the set of complex numbers. For $\rho>0$, define the kernel $C_\rho$ by $$C_\rho(z,w)=\frac{\rho}{2\pi}{\rm E}_{2/\rho,\,2/\rho}(z\bar w)\eee^{-|z|^\rho/2-|w|^\rho/2},\quad z,w\in \mathbb{C}.$$ Here, $\bar w$ denotes the complex conjugate of $w$ and, for $a,b>0$, ${\rm E}_{a,\,b}$ denotes the Mittag-Leffler function with parameters $a$ and $b$ given by $${\rm E}_{a,\,b}(z):=\sum_{k\geq 0}\frac{z^k}{\Gamma(ak+b)},\quad z\in \mathbb{C},$$ and $\Gamma$ is the Euler gamma-function. Denote by $\Theta_\rho$ a simple point process on $\mathbb{C}$ such that, for any $k\in\mn$ and any mutually disjoint Borel subsets $B_1,\ldots, B_k$ of $\mathbb{C}$, $$\me \Big[ \prod_{j=1}^k \Theta(B_j)\Big]=\int_{B_1\times\ldots\times B_k} {\rm det}(C_\rho(z_i, z_j))_{1\leq i,j\leq k}\,{\rm d}z_1\ldots {\rm d}z_k,$$ where ${\rm det}$ denotes the determinant. The point process $\Theta_\rho$ is a {\it determinantal point process} with kernel $C_\rho$ with respect to Lebesgue measure on $\mathbb{C}$. The process $\Theta_2$ (which corresponds to $\rho=2$) is known in the literature as the {\it infinite Ginibre point process}. We refer to the monograph \cite{Hough+Krishnapur+Peres+Virag:2009}, which contains a wealth of information on determinantal point processes. A discussion of the Ginibre point processes can be found in Sections 4.3.7 and 4.7 of that book and in Part I of the very recent monograph \cite{Byun+Forrester:2025}. It is stated on pp.~3-4 in \cite{Adhikari:2018} that the set of absolute values of atoms of $\Theta_\rho$ has the same distribution as $((\hat S_n)^{1/\rho})_{n\geq 1}$, where $\hat S_1$ has the gamma distribution with parameters $2/\rho$ and $1$, that is,
\begin{equation}\label{eq:gamma}
\mmp\{\hat S_1\in{\rm d}x\}=\frac{1}{\Gamma(2/\rho)}x^{2/\rho-1}\eee^{-x}\1_{(0,\infty)}(x){\rm d}x.
\end{equation}
For each $t\geq 0$, let $\Theta_\rho(D_t)$ denote the number of atoms of $\Theta_\rho$ inside the disk $D_t:=\{z\in \mathbb{C}: |z|<t\}$. Then
\begin{equation}\label{eq:principal}
(\Theta_\rho(D_t))_{t\geq 0}~~\text{has the same distribution as}~~(\hat N(t^\rho))_{t\geq 0}=\Big(\sum_{n\geq 1}\1_{\{\hat S_n\leq t^\rho\}}\Big)_{t\geq 0},
\end{equation}
with $\hat S_1$ as in \eqref{eq:gamma}. Also, we note, without going into details, that, according to Theorem 3.1 in \cite{Akemann+Strahov:2013}, the set of absolute values of atoms of a generalized infinite Ginibre point process parameterized by $m\in\mn$ has the same distribution as $((\hat S^{(1)}_n\hat S^{(2)}_{n}\cdot\ldots\cdot \hat S^{(m)}_n)^{1/2})_{n\geq 1}$, where $(\hat S^{(1)}_n)_{n\geq 1},\ldots, (\hat S^{(m)}_n)_{n\geq 1}$ are independent copies of $(\hat S_n)_{n\geq 1}$, with $\hat S_1$ having the exponential distribution of unit mean.

In \cite{Alsmeyer+Iksanov+Kabluchko:2025}, a functional limit theorem for $\hat N(t)$, properly scaled, normalized and centered, was proved under the assumption that the distribution of $\xi$ belongs to the domain of attraction of a stable distribution with finite mean. If the distribution of $\xi$ is exponential, the process $\hat N$ may be called {\it decoupled Poisson process}. A functional central limit theorem for the decoupled Poisson process was earlier obtained in Proposition 1.4 of \cite{Fenzl+Lambert:2022}. A law of the {\it single} logarithm for the decoupled Poisson process can be found in Theorem 3.1 of \cite{Buraczewski+Iksanov+Kotelnikova:2025}.

In this article we give a fairly complete picture of the asymptotic properties of the distributions of both $\max_{1\leq k\leq n}\,\hat S_k$ and $\hat \tau(t)$, properly scaled, normalized and centered. There are five  different regimes, which are regulated by the right tail of $\xi$. The first and second are treated in Theorems \ref{thm:02} and \ref{thm:02tau}; the third in Theorems \ref{thm:03} and \ref{thm:03tau}, the fourth in Theorems \ref{thm:05} and \ref{thm:05tau}, and the fifth (boundary regime) in Theorems \ref{thm:06} and \ref{thm:06tau}.

We do not see any natural applications of our present results to the point processes $\Theta_\rho$. The reason is that, as far as we understand, the absolute values of atoms of $\Theta_\rho$ cannot be enumerated in a useful way.

For an interval $I$, finite or infinite, open or closed, denote by $D(I)$ the Skorokhod space, that is, the set of c\`{a}dl\`{a}g functions defined on $I$. We shall use the standard $J_1$-topology on $D(I)$. Comprehensive information on the $J_1$-topology can be found in \cite{Billingsley:1968, Ethier+Kurtz:2005}. We write $\Longrightarrow$ to denote weak convergence in a functional space.

\begin{thm}\label{thm:02}
Assume that
\begin{equation}\label{eq:regvar}
\mmp\{\xi>v\}\sim v^{-\alpha}\ell(v),\quad v\to\infty
\end{equation}
for some $\alpha\in (0,2]$ and some $\ell$ slowly varying at $\infty$. Let $a(v)$ be any positive function satisfying $\lim_{v\to\infty}v^2\mmp\{\xi>a(v)\}=1$. Denote by $(t_k, j_k)$ the atoms of a Poisson random measure $\mathcal{P}_\alpha^{(1)}$ on $[0,\infty)\times (0,\infty]$ with mean measure $\theta\times\nu_\alpha$, where $\theta$ and $\nu_\alpha$ are measures on $[0,\infty)$ and $(0,\infty]$, respectively, defined by $$\theta ([0,\,x])=x^2/2\quad\text{and}\quad \nu_\alpha((x,\infty])=x^{-\alpha}.$$ If $\alpha\in (0,2)$ or $\alpha=2$ and $\lim_{v\to\infty}\ell(v)=\infty$, then
\begin{equation}\label{eq:basic}
\Big(\frac{\max_{1\leq k\leq \lfloor tv\rfloor}\,\hat S_k}{a(v)}\Big)_{t\geq 0}
~\Longrightarrow~
\left(\sup_{k:\,t_k \leq t}\, j_k\right)_{t\geq 0}=:(X_1(t))_{t\geq 0},\quad v\to\infty
\end{equation}
in the $J_1$-topology on $D([0,\infty))$, whereas, if $\alpha=2$ and $\liminf_{v\to\infty}\ell(v)<\infty$, then relation \eqref{eq:basic} holds with $\hat S_k-\mu k$ replacing $\hat S_k$, where $\mu:=\me [\xi]<\infty$. The one-dimensional distributions of the limit process are Fr\^{e}chet distributions given by
\begin{equation*} 
\mmp\{X_1(t)\leq y\} = \exp(-t^2 y^{-\alpha}/2),\quad y>0
\end{equation*}
and $\mmp\{X_1(t)\leq y\}=0$ for $y\leq 0$. 
\end{thm}
A realization of $(X_1(t))_{t\geq 0}$ is shown in the left panel of Figure~\ref{fig:x1-x2-realizations}.

\begin{thm}\label{thm:03}
Assume that \eqref{eq:regvar} holds either for some $\alpha\in (2,3)$ and some $\ell$ slowly varying at $\infty$ or for $\alpha=3$ and slowly varying $\ell$ satisfying $\lim_{v\to\infty}(\ell(v)/\log v)=\infty$. Let $a(v)$ be any positive function satisfying $\lim_{v\to\infty}va(v)\mmp\{\xi>a(v)\}=1$. Denote by $(t_k, j_k)$ the atoms of a Poisson random measure $\mathcal{P}_\alpha^{(2)}$ on $\mr \times (0,\infty]$ with mean measure ${\rm Leb}\times\nu_\alpha$, where ${\rm Leb}$ denotes Lebesgue measure on $\mr$ and $\nu_\alpha$ is a measure $(0,\infty]$ defined by $\nu_\alpha((x,\infty])=x^{-\alpha}$. Then 
\begin{equation}\label{eq:basic03}
\Big(\frac{\max_{1\leq k\leq \lfloor v+ ta(v)\rfloor}\,\hat S_k-\mu v}{a(v)}\Big)_{t\in\mr}
~\Longrightarrow~
\left(\sup_{k:\,t_k \leq t}\,(\mu t_k+j_k)\right)_{t\in\mr}=:(X_2(t))_{t\in\mr},\quad v\to\infty
\end{equation}
in the $J_1$-topology on $D(\mr)$, where $\mu=\me [\xi]<\infty$. 
The one-dimensional distributions of the limit process are (truncated) Fr\^{e}chet distributions given by
\begin{equation*}
\mmp\{X_2(t)\leq y\} = \exp\Big(-\frac{1}{\mu(\alpha-1)(y-\mu t)^{\alpha-1}}\Big),\quad t\in\mr,~ y>\mu t 
\end{equation*}
and $\mmp\{X_2(t)\leq y\}=0$ for $t\in\mr$ and $y\leq \mu t 
$. 
\end{thm}
A realization of $(X_2(t))_{t\in \mr}$ is shown in the right panel of Figure~\ref{fig:x1-x2-realizations}.

\begin{figure}[t]
  \centering
  \includegraphics[width=0.49\linewidth]{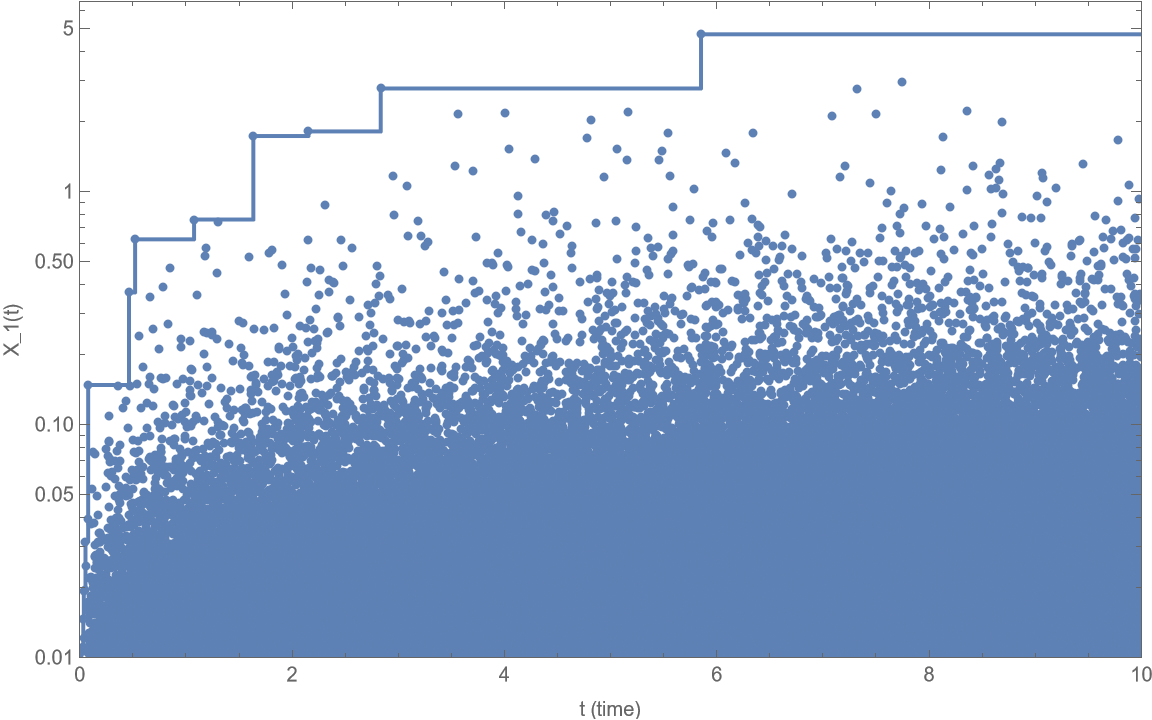}
  \includegraphics[width=0.49\linewidth]{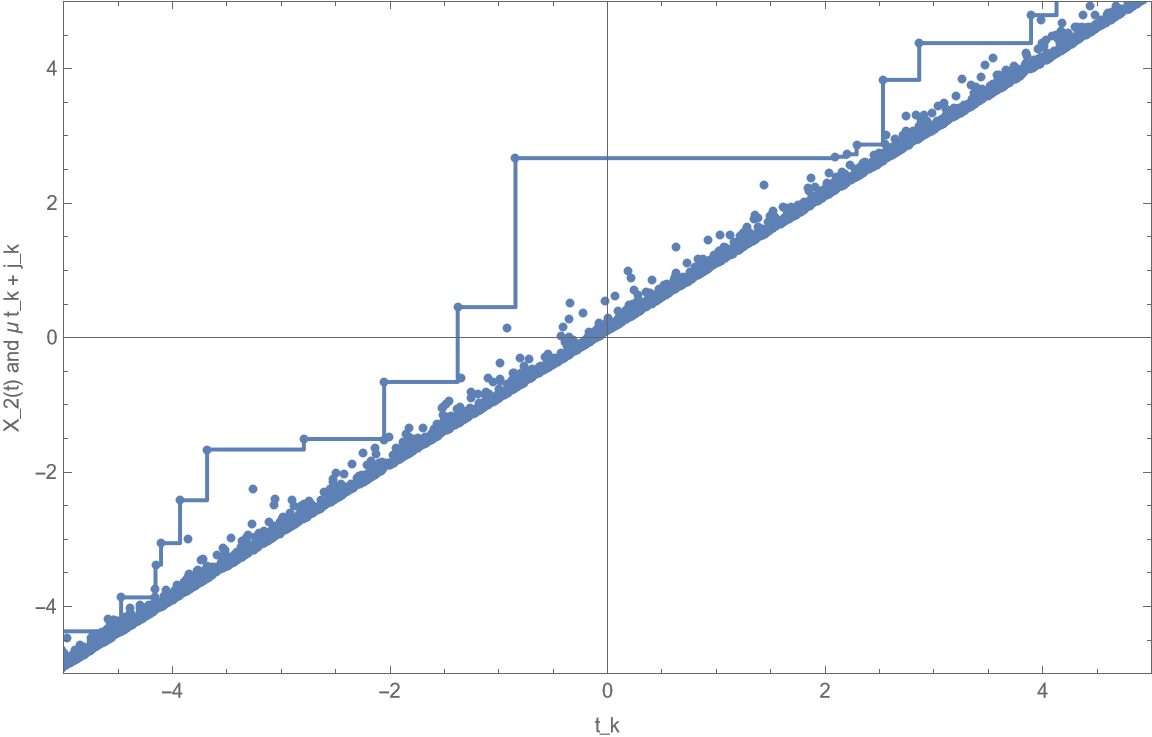}
\caption{
Left: A sample path of the process $(X_1(t))_{t\ge 0}$ from Theorem~\ref{thm:02}, together with the points $(t_k,j_k)$ of the underlying Poisson random measure $\mathcal{P}_\alpha^{(1)}$. The vertical axis is displayed on a logarithmic scale.
Right: A sample path of the process $(X_2(t))_{t\in\mathbb{R}}$ from Theorem~\ref{thm:03}, together with the points $(t_k,\mu t_k+j_k)$.
}
\label{fig:x1-x2-realizations}
\end{figure}

Let $\Phi$ denote the standard normal distribution function, that is, $$\Phi(x):=(2\pi)^{-1/2}\int_{-\infty}^x \eee^{-y^2/2}{\rm d}y,\quad x\in\mr.$$ The function $\Phi$ is strictly increasing and continuous. Hence, its inverse $\Phi^{-1}$ is well-defined.
\begin{thm}\label{thm:05}
Assume that either $\me [\xi^3]<\infty$ or \eqref{eq:regvar} holds with $\alpha=3$ and some $\ell$ slowly varying at $\infty$ which satisfies $\lim_{v\to\infty}(\ell(v)/\log v)=0$. Put $a(v):=\sigma (v/\log v)^{1/2}$ and
$$
m(v):=\sigma(\mu v+v^{1/2}\Phi^{-1}(1-1/a(v)))
$$
for $v>1$, where $\mu=\me [\xi]<\infty$ and $\sigma^2:={\rm Var}\,[\xi]\in (0,\infty)$. Denote by $(t_k, j_k)$ the atoms of a Poisson random measure $\mathcal{P}^{(3)}$ on $\mr \times \mr$ with mean measure $\rho\times\nu$, where $\rho$ and $\nu$ are measures on $\mr$ defined by $\rho({\rm d}x):=\eee^{\mu x}{\rm d}x$ and $\nu({\rm d}y):=\eee^{-y}{\rm d}y$ for $x,y\in\mr$. Then
\begin{equation}\label{eq:basic05}
\Big(\frac{\max_{1\leq k\leq \lfloor v+ ta(v)\rfloor}\,\hat S_k- m(v)}{a(v)}\Big)_{t\in\mr}
~\Longrightarrow~
\left(\sup_{k:\,t_k \leq t}\,j_k\right)_{t\in\mr}=:(X_3(t))_{t\in\mr},\quad v\to\infty
\end{equation}
in the $J_1$-topology on $D(\mr)$.
The one-dimensional distributions of the limit process are (shifted) Gumbel distributions given by
\begin{equation*}
\mmp\{X_3(t)\leq y\} = \exp\Big(-\mu^{-1}\eee^{\mu t-y}\Big),\quad t,y\in\mr.
\end{equation*}
\end{thm}
A sample path of $(X_3(t))_{t\in \mr}$ is shown in the left panel of Figure~\ref{fig:x3-x4-realizations}.
\begin{rem}\label{rem:expansion}
It can be checked that $$\Phi^{-1}(1-h)=(2\log 1/h)^{1/2}-\frac{\log\log 1/h+\log 4\pi}{(8\log 1/h)^{1/2}}+o\Big(\frac{1}{(\log 1/h)^{1/2}}\Big),\quad h\to 0+.$$ This entails $$\Phi^{-1}(1-1/a(v))~=~
(\log v)^{1/2}-\frac{\log\log v}{(\log v)^{1/2}}-\frac{\log 2\pi}{(8\log v)^{1/2}}+o\Big(\frac{1}{(\log v)^{1/2}}\Big),\quad v\to\infty.$$ Thus, the centering $m(v)$ in Theorem \ref{thm:05} could have been replaced with $$\sigma \Big(\mu v+v^{1/2}\Big((\log v)^{1/2}-\frac{\log\log v}{(\log v)^{1/2}}-\frac{\log 2\pi}{(8\log v)^{1/2}}\Big)\Big).$$
\end{rem}

Finally, we give a result dealing with a boundary case with respect to the cases treated in Theorems \ref{thm:03} and \ref{thm:05}.
\begin{thm}\label{thm:06}
Assume that $\mmp\{\xi>v\}\sim Av^{-3}\log v$ as $v\to\infty$ for some $A>0$. Put $a(v)=((A/2)v\log v)^{1/2}$ and denote by $(t_k, j_k)$ the atoms of a Poisson random measure $\mathcal{P}_3^{(2)}$ (that is, $\mathcal{P}_\alpha^{(2)}$ with $\alpha=3$) restricted to $\mr\times ((2/A)^{1/2},\infty)$. Then, as $v\to\infty$,
\begin{equation}\label{eq:basic06}
\Big(\frac{\max_{1\leq k\leq \lfloor v+ ta(v)\rfloor}\,\hat S_k-\mu v}{a(v)}\Big)_{t\in\mr}
~\Longrightarrow~
\big(\max \big\{\mu t+ (2/A)^{1/2}, \sup_{k:\,t_k \leq t}\,(\mu t_k+j_k)\big\}\big)_{t\in\mr}=:(X_4(t))_{t\in\mr}
\end{equation}
in the $J_1$-topology on $D(\mr)$, where $\mu=\me [\xi]<\infty$. The one-dimensional distributions of the limit process are (truncated) Fr\^{e}chet distributions with atoms at $\mu t+(2/A)^{1/2}$ given by
\begin{equation*}
\mmp\{X_4(t)\leq y\} = \exp\Big(-\frac{1}{2\mu (y-\mu t)^2}\Big),\quad t\in\mr,~ y\geq \mu t+(2/A)^{1/2}
\end{equation*}
and $\mmp\{X_4(t)\leq y\}=0$ for $t\in\mr$ and $y<\mu t+(2/A)^{1/2}$. 
\end{thm}
A sample path of $(X_4(t))_{t\in \mr}$ is shown in the right panel of Figure~\ref{fig:x3-x4-realizations}.

\begin{figure}[t]
  \centering
  \includegraphics[width=0.49\linewidth]{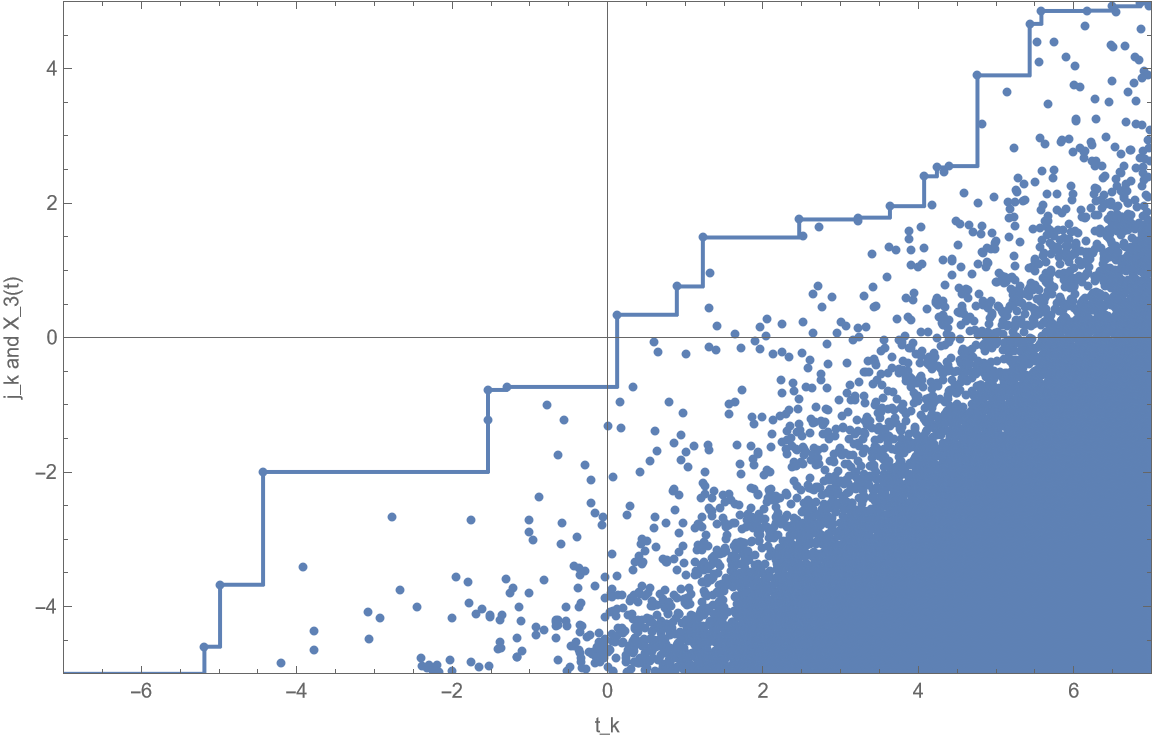}
  \includegraphics[width=0.49\linewidth]{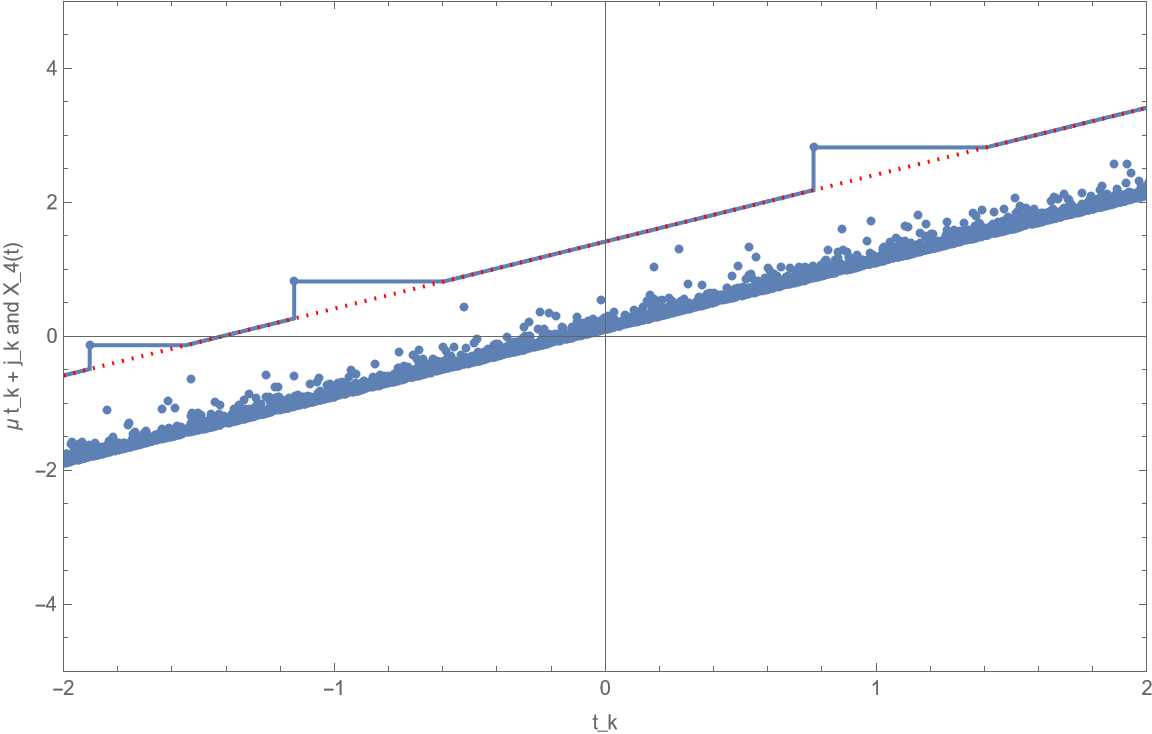}
\caption{
Left: A sample path of the process $(X_3(t))_{t\in \mr}$ from Theorem~\ref{thm:05}, together with the atoms $(t_k,j_k)$ of the underlying Poisson random measure $\mathcal{P}^{(3)}$.
Right: A sample path of the process $(X_4(t))_{t\in\mr}$ from Theorem~\ref{thm:06}, together with the points $(t_k,\mu t_k+j_k)$.
}
\label{fig:x3-x4-realizations}
\end{figure}

Now we present our results on the weak convergence of the first passage time processes. 
For a monotone nondecreasing function $f: \mr\to\mr$ denote by $f^{\leftarrow}$ its generalized inverse, that is, $f^{\leftarrow}(y):=\inf\{x\in\mr:\ f(x)>y\}$ for $y\in\mr$.

\begin{thm}\label{thm:02tau}
Under the assumptions of Theorem \ref{thm:02}, if $\alpha\in (0,2)$ or $\alpha=2$ and $\lim_{v\to\infty}\ell(v)=\infty$, then
\begin{multline}\label{eq:inverseextremal}
\big((\mmp\{\xi>v\})^{1/2}\hat \tau(tv)\big)_{t\geq 0}~\Longrightarrow~
\left(X^\leftarrow_1(t)\right)_{t\geq 0}
:=\big(\inf\{z\geq 0: \sup_{k:\,t_k \leq z}\, j_k > t\}\big)_{t\geq 0}\\=
\left(\inf\{t_k:~ j_k > t\}\right)_{t\geq 0},\quad v\to\infty
\end{multline}
in the $J_1$-topology on $D([0,\infty))$, whereas, if $\alpha=2$ and $\liminf_{v\to\infty}\ell(v)<\infty$, then relation \eqref{eq:inverseextremal} holds with $\inf\{n\geq 1: \hat S_n-\mu n>t\}$ replacing $\hat \tau(t)$.
The one-dimensional distributions of the limit process are given by
\begin{equation}\label{eq:marginaltau}
\mmp\{X^\leftarrow_1(t)\leq y\} =1-\exp(-t^{-\alpha}y^2/2),\quad y>0
\end{equation}
and $\mmp\{X^\leftarrow_1(t)\leq y\}=0$ for $y\leq 0$. Thus, the distribution of $(X^\leftarrow_1(t))^2$ is exponential.
\end{thm}

\begin{thm}\label{thm:03tau}
Under the assumptions of Theorem \ref{thm:03}, as $v\to\infty$,
\begin{multline} \label{eq:inverseextremal3}
\Big(\frac{\hat \tau(v+ta(\mu^{-1}v))-\mu^{-1}v}{\mu^{-1/(\alpha-1)}a(v)}\Big)_{t\in\mr}~\Longrightarrow~
\left(X^\leftarrow_2(t)\right)_{t\in\mr}
:=\big(\inf\{z\in\mr: \sup_{k:\,t_k \leq z}\,\big(\mu t_k+j_k\big)> t\}\big)_{t\in\mr}\\=
\left(\inf\{t_k:~
\mu t_k+j_k
> t\}\right)_{t\in\mr}
\end{multline}
in the $J_1$-topology on $D(\mr)$. The one-dimensional distributions of the limit process are given by
\begin{equation}\label{eq:marginaltau03}
\mmp\{X^\leftarrow_2(t)\leq y\} =1-\exp\Big(\frac{1}{\mu(\alpha-1)(t-\mu y)^{\alpha-1}}\Big),\quad t\in\mr,~ y<\mu^{-1} t
\end{equation}
and $\mmp\{X^\leftarrow_2(t)\leq y\}=1$ for $t\in\mr$ and $y\geq \mu^{-1}t$.
\end{thm}

\begin{thm}\label{thm:05tau}
Under the assumptions of Theorem \ref{thm:05}, as $v\to\infty$,
\begin{multline*}
\Big(\frac{\hat \tau(m(v)+ta(v))-v}{a(v)}\Big)_{t\in\mr}~\Longrightarrow~
\left(X^\leftarrow_3(t)\right)_{t\in\mr}
:=\big(\inf\{z\in\mr: \sup_{k:\,t_k \leq z}\,j_k> t\}\big)_{t\in\mr}\\=\big(\inf\{t_k:~ j_k> t\}\big)_{t\in\mr}
\end{multline*}
in the $J_1$-topology on $D(\mr)$. The one-dimensional distributions of the limit process are given by
\begin{equation*}
\mmp\{X^\leftarrow_3(t)\leq y\} =1-\exp\Big(- \mu^{-1} 
\eee^{\mu y-t}\Big),\quad t,y\in\mr.
\end{equation*}
\end{thm}

\begin{thm}\label{thm:06tau}
Under the assumptions of Theorem \ref{thm:06}, as $v\to\infty$,
\begin{multline} \label{eq:inverseextremal6}
\Big(\frac{\hat \tau(v+ta(\mu^{-1}v))-\mu^{-1}v}{\mu^{-1/2}a(v)}\Big)_{t\in\mr}~\Longrightarrow~
\left(X^\leftarrow_4(t)\right)_{t\in\mr}
\\
\ai{:=}\big(\min \big\{\mu^{-1} (t-(2/A)^{1/2} 
),\,  \inf\{t_k:~
\mu t_k+j_k 
> t\}\big\} \big)_{t\in\mr}
\end{multline}
in the $J_1$-topology on $D(\mr)$. The one-dimensional distributions of the limit process are given by
\begin{equation}\label{eq:marginaltau06}
\mmp\{X^\leftarrow_4(t)\leq y\} =1-\exp\Big(\frac{1}{2 \mu(t-\mu y)^{2}}\Big),\quad t\in\mr,~ y<\mu^{-1} (t - (2/A)^{1/2} 
)
\end{equation}
and $\mmp\{X^\leftarrow_4(t)\leq y\}=1$ for $t\in\mr$ and $y\geq \mu^{-1}(t -(2/A)^{1/2} 
)$.
\end{thm}

\section{The first passage time into $(t,\infty)$ vs the number of visits to $(-\infty, t]$}

Let $(S_n)_{n\geq 1}$ be a standard random walk, with not necessarily nonnegative jumps. For $t\in\mr$, put $$N(t):=\sum_{n\geq 1}\1_{\{S_n\leq t\}}\quad\text{and}\quad \tau(t):=\inf\{n\geq 1: S_n>t\}.$$ If $\mu=\me [S_1]\in (0,\infty)$, then, for each $t\in\mr$, these random variables are a.s.\ finite. It is known, see, for instance, Proposition A.1 in \cite{Iksanov+Kondratenko:2021}, that if $\mu\in (0,\infty)$ and $\sigma^2={\rm Var}\,[S_1]\in (0,\infty)$, then $$\Big(\frac{M(vt)-\mu^{-1}vt}{(\sigma^2\mu^{-3}t)^{1/2}}\Big)_{t\geq 0}~\Longrightarrow~(B(t))_{t\geq 0},\quad v\to\infty$$ in the $J_1$-topology on $D([0,\infty))$. Here, $M$ is either $N$ or $\tau$, and $B$ is a standard Brownian motion. In words, the distributional asymptotic behaviors of the number of visits to $(-\infty, t]$ and the first passage time into $(t,\infty)$ are identical for standard random walks with a positive mean and a finite variance. This is not totally surprising, since the major contribution to $N(t)$ is made by the positions $S_k$ with $k\leq \tau(t)-1$, whereas the contribution of the subsequent positions is $o(t^{1/2})$. The argument is particularly simple in the case of one-dimensional distributions: just write $$N(t)-(\tau(t)-1)=\sum_{n\geq \tau(t)+1}\1_{\{S_n\leq t\}}\leq \sum_{n\geq 1}\1_{\{S_{n+\tau(t)}-S_{\tau(t)}\leq 0\}}\quad\text{a.s.}$$ and note that the latter variable has the same distribution as $N(0)$.

As far as decoupled standard random walks are considered, the situation is drastically different. This change is caused by the fact that the sequence $(\hat S_n)_{n\geq 1}$ is not monotone nondecreasing. 
This leads to the following effect: the number of visits $\hat N$ and the first passage time $\hat \tau$ exhibit absolutely different distributional behaviors. For instance, according to Theorem 2.1 in \cite{Alsmeyer+Iksanov+Kabluchko:2025}, under the assumption that the distribution of $\xi$ belongs to the domain of attraction of a stable distribution with finite mean, the limit processes for $\hat N$ are stationary Gaussian. On the other hand, it follows from the results of the previous section that the limit processes for $\hat \tau$ are inverse extremal-like processes. Specifically, assume that $\me [\xi^3]<\infty$. By formula (9) in \cite{Alsmeyer+Iksanov+Kabluchko:2025}, $$\Big(\frac{\hat N(\sigma^2(t+v)^2)-\mu^{-1}\sigma^2(t+v)^2}{(\mu^{-3/2}\sigma^2 v)^{1/2}}\Big)_{t\geq 0}~\Longrightarrow~ (W(t))_{t\geq 0},\quad v\to\infty$$ in the $J_1$-topology on $D([0,\infty))$ provided that the function $x\mapsto \sum_{n\geq 1}\mmp\{S_n\leq x\}$ is Lipschitz continuous. Here, $W$ is a centered stationary Gaussian process with explicitly known covariance. A very different distributional limit theorem for $\hat \tau$ can be found in Theorem \ref{thm:03tau}. Again, the striking difference outlined above is not surprising. The asymptotic behaviors of $\hat N$ and $\hat \tau$ are driven by the variables $\hat S_k$ with $k=\lfloor \mu^{-1}v+tv^{1/2}\rfloor$ as $t$ varies over $[0,\infty)$ and $k=\lfloor v+t(v/\log v)^{1/2}\rfloor$ as $t$ varies over $\mr$, respectively.

\section{Proofs of the main results}
\subsection{Large deviation probabilities in the case
$\alpha\le 2$}
\begin{lemma}\label{lem:tech}
Assume that  \eqref{eq:regvar} holds with some $\alpha\le 2$.
Then, for fixed $t,y>0$,
\begin{equation}\label{eq:limit}
\lim_{v\to\infty} v\mmp\{S_{\lfloor tv\rfloor}>a(v)y\}=ty^{-\alpha}
\end{equation}
if $\me [\xi]=\infty$ and
\begin{equation}\label{eq:limit1}
\lim_{v\to\infty} v\mmp\{S_{\lfloor tv\rfloor}-\mu \lfloor tv\rfloor>a(v)y\}=ty^{-\alpha}
\end{equation}
if $\mu=\me [\xi]<\infty$.
\end{lemma}
\begin{proof}
It is a standard fact (see, for instance, Lemma 6.1.3 in \cite{Iksanov:2016}) that $a$ is regularly varying at $\infty$ of index $2/\alpha$. In particular, $a 
$ grows faster than the normalizing sequence 
in the limit theorem for $S_n$. Then the desired relations in \eqref{eq:limit} and \eqref{eq:limit1} in the case 
$\alpha<2$ follow from Theorem~2.1 in 
~\cite{Berger2019}.

Now we treat 
the case $\alpha=2$.
We start by deriving lower bounds for $\mmp\{S_n^0>x\}$, where $S_n^0:=S_n-\mu n$ for $n\in\mn$. Put 
also $\xi_k^0:=\xi_k-\mu$ for $k\in\mn$. For each $\varepsilon>0$ and $x>0$, 
\begin{align*}
 \mmp\{S_n^0>x\}
 &\ge n\mmp\{\xi_1^0>(1+\varepsilon)x\}
 \mmp\big\{S_{n-1}^0\ge -\varepsilon x,\max_{j\le n-1}\xi_j^0\le (1+\varepsilon)x\big\}\\
 &\ge\mmp\{\xi_1^0>(1+\varepsilon)x\}\big(1-\mmp\{S_{n-1}^0<-\varepsilon x\}-\mmp\big\{\max_{j\le n-1}\xi_j^0> (1+\varepsilon)x\big\}\big)\\
 &\ge\mmp\{\xi_1^0>(1+\varepsilon)x\}\left(1-n\mmp\{\xi_1^0>(1+\varepsilon)x\}-\mmp\{S_{n-1}^0<-\varepsilon x\}\right).
\end{align*}
Since $\xi_k^0\ge -\mu$, an application of the Bernstein inequality yields 
$$
\mmp\{S_{n-1}^0<-\varepsilon x\}\to0
$$
provided that $xn^{-1/2} 
\to\infty$. This leads to the relation
$$
\liminf_{n\to\infty}\frac{\mmp\{S_n^0>x\}}{n\mmp\{\xi_1^0>x\}}
\ge (1+\varepsilon)^{-\alpha}
$$
for $x$ satisfying $xn^{-1/2} 
\to\infty$ and $n\mmp\{\xi_1>x\}\to0.$ Letting here $\varepsilon\to0+$, we obtain
\begin{align}
\label{eq:lower-b}
 \liminf_{n\to\infty}\frac{\mmp\{S_n^0>x\}}{n\mmp\{\xi_1^0>x\}}
\ge 1
\end{align}
for $x$ satisfying $xn^{-1/2} 
\to\infty$ and $n\mmp\{\xi_1>x\}\to0.$

To obtain the corresponding upper bound we apply one of the Fuk-Nagaev inequalities. By Theorem~1.2 in 
\cite{Nagaev:1979}, with $t=2$ and $y=x/r 
$ for 
some $r>1$, 
$$
\mmp\{S_n^0>x\}
\le n \mmp\{\xi_1^0>x/r\}
+\eee^r\left(\frac{rn\sigma^2(x/r)}{x^2}\right)^r,
$$
where
$$
\sigma^2(y):=\me\left[(\xi_1^0)^2\1_{\{|\xi_1^0|\leq y\}} 
\right],\quad y\geq 0.
$$
The assumption \eqref{eq:regvar} ensures 
that the function
$\sigma^2 
$ is slowly 
varying at infinity. Consequently, there exists a constant $C_r$ such that
$$
\mmp\{S_n^0>x\}
\le n \mmp\{\xi_1^0>x/r\}
+C_r\left(\frac{n\sigma^2(x)}{x^2}\right)^r.
$$
This bound implies that if $x\ge n^{1/2+\delta}$ with some
$\delta>0$, then
$$
\limsup_{n\to\infty}
\frac{\mmp\{S_n^0>x\}}{n \mmp\{\xi_1^0>x\}}\le r^\alpha.
$$
Letting now $r\to1+$, we conclude that
\begin{align}
\label{eq:upper-b}
\limsup_{n\to\infty}
\frac{\mmp\{S_n^0>x\}}{n \mmp\{\xi_1^0>x\}}\le 1
\end{align}
provided that $x\ge n^{1/2+\delta}$.
Combining \eqref{eq:lower-b} and \eqref{eq:upper-b} and recalling that $a
$ is regularly varying of index $2/\alpha$, we infer that \eqref{eq:limit1} is valid for $\alpha=2$. Thus, the proof of Lemma \ref{lem:tech} is complete.
\end{proof}
Now we argue that the centering $\mu\lfloor tv\rfloor$ in \eqref{eq:limit1} may be omitted in some cases. Specifically, if $\alpha\in (1,2)$ or $\alpha=1$ and $\mu<\infty$, then regular variation of $a$ of index $2/\alpha$ entails $\lim_{v\to\infty} v^{-1}a(v)=\infty$. The latter limit relation also holds true if $\alpha=2$ and $\lim_{v\to\infty}\ell(v)=\infty$ as a consequence of $(\ell(a(v)))^{1/2}\sim v^{-1}a(v)$ as $v\to\infty$. Summarizing, if $\alpha\in (0,2)$ or $\alpha=2$ and $\lim_{v\to\infty}\ell(v)=\infty$, then \eqref{eq:limit} holds, whereas if $\alpha=2$ and $\liminf_{v\to\infty}\ell(v)<\infty$, then \eqref{eq:limit1} holds.

\subsection{Proof of Theorem \ref{thm:02}}

For intervals $I$ and $J$, denote by $M_p(I\times J)$ the set of Radon point measures on $I\times J$ endowed with the vague topology. Also, let $\varepsilon_{(t,x)}$ denote the Dirac measure with atom at $(t,x)$. As a preparation, we need a result on vague convergence. Recall that the definition of the Poisson random measure $\mathcal{P}_\alpha^{(1)}$ was given in Theorem~\ref{thm:02}.
\begin{assertion}\label{prop:vague}
Under the assumptions of Theorem \ref{thm:02}, if $\alpha\in (0,2)$ or $\alpha=2$ and $\lim_{v\to\infty}\ell(v)=\infty$, then $$\sum_{k\geq 1}\varepsilon_{(k/v, \hat S_k/a(v))}~\Longrightarrow~ \mathcal{P}_\alpha^{(1)},\quad v\to\infty$$ on $M_p([0,\infty)\times (0,\infty])$, whereas if $\alpha=2$ and $\liminf_{v\to\infty}\ell(v)<\infty$, then $$\sum_{k\geq 1}\varepsilon_{(k/v, (\hat S_k-\mu k)/a(v))}\1_{\{\hat S_k>\mu k\}}~\Longrightarrow~ \mathcal{P}_\alpha^{(1)},\quad v\to\infty$$ on $M_p([0,\infty)\times (0,\infty])$, where $\mu=\me [\xi]<\infty$.  \ref{thm:02}.
\end{assertion}
\begin{proof}

\noindent {\sc The case where \eqref{eq:limit} prevails.} According to Proposition 3.19 on p.~153 in \cite{Resnick:1987}, it suffices to prove that
\begin{equation}\label{eq:princ}
\lim_{v\to\infty} \me\Big[\exp\Big(-\sum_{k\geq 1}f\big(k/v, \hat S_k/a(v)\big)\Big)\Big]=\exp\Big(-\int_{[0,\infty)\times (0,\infty)} (1-\eee^{-f(t,y)})\theta({\rm d}t)\nu_\alpha({\rm d}y)\Big)
\end{equation}
for any nonnegative continuous function $f$ on $[0,\infty)\times (0,\infty]$ with compact support, with the convention that the sets $[0, a]\times [b,\infty]$ are compact on $[0,\infty)\times (0,\infty]$ for $a,b>0$.

Fix any $f$ as above. Then there exist $\lambda, \gamma>0$ such that $f(t,y)=0$ whenever either $t>\lambda$ or $y\in [0,\gamma)$. By the independence of $\hat S_1$, $\hat S_2,\ldots$, 
\begin{multline*}
\log \me\Big[\exp\Big(-\sum_{k\geq 1}f(k/v, \hat S_k/a(v))\Big)\Big]=-\sum_{k\geq 1}\log \me\big[\exp(-f(k/v, S_k/a(v)))\big]\\=-\sum_{k\geq 1}\log \Big(1-\int_{(\gamma,\infty)}(1-\eee^{-f(k/v,y)})\mmp\{S_k/a(v)\in {\rm d}y\}\Big).
\end{multline*}
Put $S_0:=0$ and define $f(0,0)$ to be $0$. Using this when passing to the integral we obtain
\begin{multline*}
\sum_{k\geq 1}\int_{(\gamma,\infty)}(1-\eee^{-f(k/v,y)})\mmp\{S_k/a(v)\in {\rm d}y\}\\=\int_0^\lambda \int_{(\gamma,\infty)}(\eee^{-f(t,y)}-\eee^{-f(\lfloor tv\rfloor/v,y)})v\mmp\{S_{\lfloor tv\rfloor}/a(v)\in {\rm d}y\}{\rm d}t\\+\int_0^\lambda \int_{(\gamma,\infty)}(1-\eee^{-f(t,y)})v\mmp\{S_{\lfloor tv\rfloor}/a(v)\in {\rm d}y\}{\rm d}t=:I(v)+J(v).
\end{multline*}
For $\delta>0$, put $\omega_{\eee^{-f}}(\delta):=\sup_{||{\bf x}-{\bf y}||\leq \delta}|\eee^{-f({\bf x})}-\eee^{-f({\bf y})}|$, where ${\bf x}, {\bf y}\in [0,\infty)\times (0,\infty)$. Then $$|I(v)|\leq \lambda \omega_{\eee^{-f}}(\sup_{t\in [0,\lambda]}\,(t-\lfloor tv\rfloor/v))v\mmp\{S_{\lfloor \lambda v\rfloor}>a(v)\gamma\}~\to~0,\quad v\to\infty.$$ This follows from \eqref{eq:limit}, $\lim_{v\to \infty}\sup_{t\in [0,\lambda]}\,(t-\lfloor vt\rfloor/v)=0$ and uniform continuity of $\eee^{-f}$. Relation \eqref{eq:limit} ensures that, with $t>0$ fixed,
$$\lim_{v\to\infty} \int_{(\gamma,\infty)}\big(1-\eee^{-f(t,y)}\big)v\mmp\{S_{\lfloor tv\rfloor}/a(v)\in {\rm d}y\}=t\int_{(0,\infty)}\big(1-\eee^{-f(t,y)}\big)\nu_\alpha({\rm d}y).
$$
For each $t\in [0,\lambda]$, there exists a constant $C>0$ such that, for all $v>0$, $$\int_{(\gamma,\infty)}\big(1-\eee^{-f(t,y)}\big)v\mmp\{S_{\lfloor tv\rfloor}/a(v)\in {\rm d}y\}\leq v\mmp\{S_{\lfloor \lambda v\rfloor}>a(v)\gamma\}\leq {\rm C},
$$
where the last inequality is justified by \eqref{eq:limit}. With this at hand, an application of the Lebesgue dominated convergence theorem yields
\begin{multline}\label{eq:inter}
\lim_{v\to\infty}\sum_{k\geq 1}\int_{(\gamma,\infty)}(1-\eee^{-f(k/v,y)})\mmp\{S_k/a(v)\in {\rm d}y\}=\lim_{v\to\infty}J(v)\\=
\int_{[0,\infty)\times(0,\infty)}(1-\eee^{-f(t,y)})\theta({\rm d}t)\nu_\alpha({\rm d}y).
\end{multline}

We claim that
\begin{equation}\label{eq:neglig}
\lim_{v\to\infty}\sup_{k\geq 1}\int_{(\gamma,\infty)}\big(1-\eee^{-f(k/v,y)}\big)\mmp\{S_k/a(v)\in {\rm d}y\}=0.
\end{equation}
Then the inequality
$$x\leq -\log (1-x)\leq x+x^2,\quad x\in [0,1/2]$$ applies with $x=\int_{(\gamma,\infty)}\big(1-\eee^{-f(k/v,y)}\big)\mmp\{S_k/a(v)\in {\rm d}y\}$ for large $v$. As a consequence of \eqref{eq:inter} and \eqref{eq:neglig} we infer
$$\lim_{v\to\infty}\sum_{k\ge 1}\Big(\int_{(\gamma,\infty)}\big(1-\eee^{-f(k/v,y)}\big)\mmp\{S_k/a(v)\in {\rm d}y\}\Big)^2=0,$$ thereby completing the proof of \eqref{eq:princ}.

\noindent {\sc Proof of \eqref{eq:neglig}.} The integral is equal to $0$ for $k\geq \lfloor\lambda v\rfloor+1$. For positive integer $k\leq \lfloor\lambda v\rfloor$, $$\int_{(\gamma,\infty)}\big(1-\eee^{-f(k/v,y)}\big)\mmp\{S_k/a(v)\in {\rm d}y\}\leq \mmp\{S_{\lfloor\lambda v\rfloor}>a(v)\gamma\}~\to~0,\quad v\to\infty$$ in view of \eqref{eq:limit}.

\noindent {\sc The case where \eqref{eq:limit1} prevails.} Invoking once again Proposition 3.19 on p.~153 in \cite{Resnick:1987} we conclude that it is sufficient to prove that
\begin{multline*}
\lim_{v\to\infty} \me\Big[\exp\Big(-\sum_{k\geq 1}f\big(k/v, (\hat S_k-\mu k)/a(v)\big)\1_{\{\hat S_k>\mu k\}}\Big)\Big]\\=\exp\Big(-\int_{[0,\infty)\times (0,\infty)} (1-\eee^{-f(t,y)})\theta({\rm d}t)\nu_\alpha({\rm d}y)\Big)
\end{multline*}
for any nonnegative continuous function $f$ on $[0,\infty)\times \mr$ with compact support in $[0,\infty)\times (0,\infty]$ (in particular, $f(x,y)=0$ for $y\leq 0$). The proof of this mimics that of \eqref{eq:princ}, the only difference being that now we appeal to \eqref{eq:limit1} rather than \eqref{eq:limit}.

The proof of Proposition \ref{prop:vague} is complete.
\end{proof}

We are ready to prove Theorem \ref{thm:02}.
\begin{proof}[Proof of Theorem \ref{thm:02}]
Recall that $\mathcal{P}_\alpha^{(1)}$ denotes a Poisson random measure with atoms $(t_k, j_k)$. Plainly, $\mathcal{P}_\alpha^{(1)}([0,\infty)\times (-\infty, 0])=\mathcal{P}_\alpha^{(1)}(\{0\}\times (0,\infty])=0$ a.s.,  $\mathcal{P}_\alpha^{(1)}$ does not have clustered jumps a.s.\ and that, for all $t,s\geq 0$, $t<s$ and all $\delta>0$, $\mathcal{P}_\alpha^{(1)}([0,t]\times [\delta,\infty])<\infty$ a.s.\ and  $\mathcal{P}_\alpha^{(1)}((t,s)\times (0,\infty))\geq 1$ a.s. Thus, $\mathcal{P}_\alpha^{(1)}$ satisfies with probability one all the assumptions imposed on $\rho_0$ in Proposition \ref{deterministic1}.
The weak convergence stated in \eqref{eq:basic} is then an immediate consequence of Proposition \ref{prop:vague} in combination with the Skorokhod representation theorem (which ensures that there are versions of the point measures from Proposition \ref{prop:vague} converging a.s.) and Proposition \ref{deterministic1} with $f(x)=0$ for $x\geq 0$.

To complete the proof, we point out the marginal distributions of the limit process $X_1$: for $y\geq 0$ and $t>0$,
\begin{multline*}
\mmp\big\{\sup_{k:~t_k \leq t}\, j_k\leq y\big\}=\mmp\big\{\mathcal{P}_\alpha^{(1)}\big((s,x): s\leq t, x>y\big)=0\big\}\\=\exp\big(-\me \big[\mathcal{P}_\alpha^{(1)} \big((s,x): s\leq t, x>y\big)\big]\big)=\exp(-t^2y^{-\alpha}/2).
\end{multline*}
\end{proof}

\subsection{Proof of Theorem \ref{thm:03}}

We start with a proposition.
\begin{assertion}\label{prop:vague3}
Under the assumptions of Theorem \ref{thm:03}, $$\sum_{k\geq 1}\varepsilon_{((k-v)/a(v), (\hat S_k-\mu k)/a(v))}\1_{\{\hat S_k>\mu k\}}~\Longrightarrow~ \mathcal{P}_\alpha^{(2)},\quad v\to\infty$$ on $M_p(\mr \times (0,\infty])$,
where $\mu=\me [\xi]<\infty$ and $\mathcal{P}_\alpha^{(2)}$ is as defined in Theorem~\ref{thm:03}.
\end{assertion}
\begin{proof}
We start by observing that $a(v)=v^{1/(\alpha-1)}L(v)$ with $L$ slowly varying at $\infty$. Now we show that
\begin{equation}\label{eq:inter23232}
\lim_{v\to\infty}\frac{(L(v))^2}{\log v}=\infty
\end{equation}
provided that $\alpha=3$ and $\lim_{v\to\infty}(\ell(v)/\log v)=\infty$. Indeed, by the definition of $a$, as $v\to\infty$, $$1~\leftarrow~va(v)\mmp\{\xi>a(v)\}~\sim~v^{3/2}L(v)\frac{\ell(v^{1/2}L(v))}{v^{3/2}(L(v))^3}~\sim~\frac{\ell(v^{1/2}L(v))}{(L(v))^2}.$$ Thus, given $\varepsilon>0$ the inequality $\ell(v^{1/2}L(v)) \leq (1+\varepsilon)(L(v))^2$ holds for large $v$. Also, the assumption concerning $\ell$ ensures that given $B>0$ the inequality $\ell(v)\geq B\log v$ holds for large $v$. Hence, $(L(v))^2\geq (1+\varepsilon)^{-1}B\log (v^{1/2}L(v))\geq (1+\varepsilon)^{-1}(B/2)\log v$ for large $v$, which proves \eqref{eq:inter23232}. Therefore, under the conditions of Theorem~~\ref{thm:03}, 
\begin{equation}\label{eq:large_a}
\lim_{v\to\infty}\frac{a(v)}{(v\log v)^{1/2} 
}=\infty.
\end{equation}
Applying Theorem 1.9 in \cite{Nagaev:1979} to the walk $(S_n-\mu n)_{n\geq 1}$, we infer 
$$\mmp\{S_n-\mu n>x\}~\sim~ n\mmp(\xi>x)$$
provided that $x>(\alpha-2+\delta)^{1/2}(n\log n)^{1/2} 
$ for some $\delta>0$. Using this with $n=\lfloor v+ ta(v)\rfloor$ and recalling the definition of $a 
$, we obtain
\begin{equation}\label{eq:limit2}
\lim_{v\to\infty} a(v)\mmp\{S_{\lfloor v+ ta(v)\rfloor}-\mu\lfloor v+ta(v)\rfloor>a(v)y\}=y^{-\alpha}.
\end{equation}

As in the proof of Proposition~\ref{prop:vague}, it suffices to show that
\begin{multline*}
\lim_{v\to\infty} \me\Big[\exp\Big(-\sum_{k\geq 1}f\big((k-v)/a(v), (\hat S_k-\mu k)/a(v)\big)\1_{\{\hat S_k>\mu k\}}\Big)\Big]\\=\exp\Big(-\int_{\mr \times (0,\infty)} (1-\eee^{-f(t,y)}){\rm d}t\, \nu_\alpha({\rm d}y)\Big)
\end{multline*}
for any nonnegative continuous function $f$ on $\mr\times \mr$ with compact support in $\mr\times (0,\infty]$ (that is, $f(x,y)=0$ for $y\leq 0$). 
For fixed $f$, there exist $\lambda, \gamma>0$ such that $f(t,y)=0$ whenever either $|t|>\lambda$ or $y<\gamma$. Using this and the independence of random variables
$\hat S_1$, $\hat S_2,\ldots$ we obtain
\begin{multline*}
\log \me\Big[\exp\Big(-\sum_{k\geq 1}f((k-v)/a(v), (\hat S_k-\mu k)/a(v))\1_{\{\hat S_k>\mu k\}}\Big)\Big]\\=-\sum_{k\geq 1}\log\big(\me\big[\exp(-f((k-v)/a(v), (S_k-\mu k)/a(v)))\1_{\{S_k>\mu k\}}\big]+\mmp\{S_k\leq \mu k\}\big)\\=-\sum_{k\geq 1}\log \Big(1-\int_\mr (1-\eee^{-f((k-v)/a(v),y)})\mmp\{(S_k-\mu k)/a(v)\in {\rm d}y\}\Big)\\=-\sum_{k\geq 1}\log \Big(1-\int_{(\gamma,\infty)}(1-\eee^{-f((k-v)/a(v),y)})\mmp\{(S_k-\mu k)/a(v)\in {\rm d}y\}\Big).
\end{multline*}
Furthermore, for large $v$,
\begin{multline*}
\sum_{k\geq 1}\int_{(\gamma,\infty)}(1-\eee^{-f((k-v)/a(v),y)})\mmp\{(S_k-\mu k)/a(v)\in {\rm d}y\}\\=\int_{-\lambda}^\lambda \int_{(\gamma,\infty)}(\eee^{-f(t,y)}-\eee^{-f((\lfloor v+ta(v)\rfloor-v)/a(v),y)})a(v)\mmp\{(S_{\lfloor v+ta(v)\rfloor}-\mu \lfloor v+ta(v)\rfloor)/a(v)\in {\rm d}y\}{\rm d}t\\+\int_{-\lambda}^\lambda \int_{(\gamma,\infty)}(1-\eee^{-f(t,y)})a(v)\mmp\{(S_{\lfloor v+ta(v)\rfloor}-\mu\lfloor v+ta(v)\rfloor)/a(v)\in {\rm d}y\}{\rm d}t=:I(v)+J(v).
\end{multline*}
Relations \eqref{eq:limit} and $\lim_{v\to \infty}\sup_{t\in [-\lambda,\lambda]}\,(t-(\lfloor v+ta(v)\rfloor-v)/a(v))=0$ together with uniform continuity of $\eee^{-f}$ entail
\begin{multline*}
|I(v)|\\\leq \omega_{\eee^{-f}}(\sup_{t\in [-\lambda,\lambda]}\,(t-(\lfloor v+ta(v)\rfloor-v)/a(v)))\int_{-\lambda}^\lambda a(v)\mmp\{S_{\lfloor v+ta(v)\rfloor}-\mu \lfloor v+ta(v)\rfloor>a(v)\gamma\}{\rm d}t\\~\to~0,\quad v\to\infty.
\end{multline*}
We omit a proof of the limit relation
$$\lim_{v\to\infty}J(v)=\int_{\mr \times (0,\infty)} (1-\eee^{-f(t,y)}){\rm d}t\, \nu_\alpha({\rm d}y).$$ It can be justified along the lines of the proof of the corresponding fragment of Proposition \ref{prop:vague}.

Referring back to the proof of Proposition \ref{prop:vague} we conclude that a counterpart of \eqref{eq:neglig} has to be obtained, whereas the remainder of the proof mimics that of Proposition \ref{prop:vague}. Thus, we are left with showing that
\begin{equation*} 
\lim_{v\to\infty}\sup_{k\geq 1}\int_{(\gamma,\infty)}\big(1-\eee^{-f((k-v)/a(v),y)}\big)\mmp\{(S_k-\mu k)/a(v)\in {\rm d}y\}=0.
\end{equation*}
This is a consequence of
\begin{multline*}
\sup_{v-\lambda a(v)\leq k\leq v+\lambda a(v)}\,\int_{(\gamma,\infty)}\big(1-\eee^{-f((k-v)/a(v),y)}\big)\mmp\{(S_k-\mu k)/a(v)\in {\rm d}y\}\\\leq \mmp\Big\{\frac{\sup_{1 \leq k\leq v+\lambda a(v)}\,\big(S_k-\mu k\big)}{v^{1/2}}>\frac{\gamma a(v)}{v^{1/2}}\Big\} ~\to~0,\quad v\to\infty.
\end{multline*}
Here, the last relation is secured by $\lim_{v\to\infty} v^{-1/2}a(v)=\infty$ and the weak convergence as $v\to\infty$ of $\frac{\sup_{1 \leq k\leq v+\lambda a(v)}\,(S_k-\mu k)}{({\rm Var}\,[\xi]v)^{1/2}}$ to the absolute value of a random variable with the standard normal distribution.

The proof of Proposition \ref{prop:vague3} is complete.
\end{proof}

\begin{proof}[Proof of Theorem \ref{thm:03}]
We start with a representation $$\frac{\max_{1\leq k\leq v+ta(v)}\,\hat S_k-\mu v}{a(v)}=\max_{k:~(k-v)/a(v)\leq t}\,\Big(\frac{\mu (k-v)}{a(v)}+\frac{\hat S_k-\mu k}{a(v)}\Big),$$ which holds for each $t\in\mr$ and sufficiently large $v$. In view of this, we intend to apply Proposition \ref{deterministic1} with $f(x)=\mu x$ for $x\in\mr$ and $\rho_n=\sum_k \varepsilon_{\{(k-v_n)/a(v_n), (\hat S_k-\mu k)/a(v_n)\}}$, where $(v_n)_{n\geq 1}$ is any sequence of positive numbers diverging to $\infty$. It can be checked that $\mathcal{P}_\alpha^{(2)}$ satisfies with probability one all the assumptions imposed on $\rho_0$ in Proposition \ref{deterministic1}. The weak convergence stated in \eqref{eq:basic03} then follows from Proposition \ref{prop:vague3} in combination with the Skorokhod representation theorem, and Proposition \ref{deterministic1}.

Finally, we discuss the marginal distributions of the limit process $X_2$: for $t\in\mr$ and $y>\mu t$,
\begin{multline*}
\mmp\big\{\sup_{k:~t_k \leq t}\,(\mu t_k+ j_k)\leq y\big\}=\mmp\big\{\mathcal{P}_\alpha^{(2)}\big((s,x): s\leq t, \mu s+x>y\big)=0\big\}\\=\exp\big(-\me \big[\mathcal{P}_\alpha^{(2)}\big((s,x): s\leq t, \mu s+x>y\big)\big]\big). 
\end{multline*}
Further,
\begin{multline*}
\me \big[\mathcal{P}_\alpha^{(2)}\big((s,x): s\leq t, \mu s+x>y\big)\big]=\int_{-\infty}^t\int_{[0,\,\infty)}\1_{\{\mu s+x>y\}}\nu_\alpha({\rm d}x){\rm d}s=\int_{-\infty}^t (y-\mu s)^{-\alpha}{\rm d}s\\=(\mu(\alpha-1))^{-1}(y-\mu t)^{1-\alpha}.
\end{multline*}
\end{proof}

\subsection{Proof of Theorem \ref{thm:05}}

\begin{assertion}\label{prop:vague5}
Under the assumptions of Theorem \ref{thm:05}, $$\sum_{k\geq 1}\varepsilon_{((k-v)/a(v), (\hat S_k-m(v))/a(v))}~\Longrightarrow~ \mathcal{P}^{(3)},\quad v\to\infty$$ on $M_p(\mr \times (-\infty,\infty])$.
\end{assertion}
\begin{proof}

Without loss of generality, we can and do assume that $\sigma^2=1$.

We claim that, for fixed $t,y\in\mr$,
\begin{equation}\label{eq:inter3}
\lim_{v\to\infty}a(v)\mmp\{S_{\lfloor v+ta(v)\rfloor}>m(v)+ya(v)\}=\eee^{\mu t-y}.
\end{equation}
Assume first that $\mmp\{\xi>v\}\sim v^{-3}\ell(v)$ as $v\to\infty$ and $\lim_{v\to\infty}(\ell(v)/\log v)=0$. According to Theorem 1.9 in \cite{Nagaev:1979},
\begin{equation}\label{eq:nagaev}
\mmp\{S_n-\mu n>x\}=\big(1-\Phi(xn^{-1/2})\big)(1+o(1))+n\mmp\{\xi>x\}(1+o(1)),\quad n\to\infty
\end{equation}
uniformly in $x\geq n^{1/2}$. Using this with
\begin{equation}\label{eq:choice}
n=\lfloor v+ta(v)\rfloor\quad\text{and}\quad x=\frac{m(v)-\mu \lfloor v+ta(v)\rfloor+y a(v)}{(\lfloor v+ta(v)\rfloor)^{1/2}},
\end{equation}
we intend to show that the normal distribution tail dominates. To this end, we first prove that, for fixed $t,y\in\mr$,
\begin{equation}\label{eq:normal}
\lim_{v\to\infty}a(v)\Big(1-\Phi\Big(\frac{m(v)-\mu\lfloor v+ta(v)\rfloor+ya(v)}{(\lfloor v+ta(v)\rfloor)^{1/2}}\Big)\Big)
=\eee^{\mu t-y}.
\end{equation}
Indeed, as a consequence of $1-\Phi(x)\sim (2\pi)^{-1/2}x^{-1}\eee^{-x^2/2}$ as $x\to\infty$ and
\begin{equation}\label{eq:asymp}
\frac{m(v)-\mu v}{v^{1/2}}=\Phi^{-1}(1-1/a(v))~\sim~ (\log v)^{1/2},\quad v\to\infty
\end{equation}
(see Remark \ref{rem:expansion}) we obtain
\begin{multline*}
1-\Phi\Big(\frac{m(v)-\mu\lfloor v+ta(v)\rfloor+ya(v)}{(\lfloor v+ta(v)\rfloor)^{1/2}}\Big)~\sim~1-\Phi\Big(\frac{m(v)-\mu v}{v^{1/2}}+\frac{(y-\mu t)a(v)}{v^{1/2}}\Big)\\=
1-\Phi\Big(\Phi^{-1}(1-1/a(v))+\frac{(y-\mu t)a(v)}{v^{1/2}}\Big)\\~\sim~\frac{1}{(2\pi)^{1/2}\Phi^{-1}(1-1/a(v))}\eee^{-(\Phi^{-1}(1-1/a(v)))^2/2}\eee^{\Phi^{-1}(1-1/a(v))(\mu t-y)(\log v)^{-1/2}}\\~\sim~ 1-\Phi(\Phi^{-1}(1-1/a(v)))\eee^{\mu t-y}=\frac{\eee^{\mu t-y}}{a(v)},\quad v\to\infty.
\end{multline*}
On the other hand, using \eqref{eq:asymp} we conclude that
\begin{multline*}
a(v)v\mmp\{\xi>m(v)-\mu\lfloor v+ta(v)\rfloor+ya(v)\}~\sim~a(v)v\mmp\{\xi>m(v)-\mu v \}~\sim~\frac{v^{3/2}}{(\log v)^{1/2}}\frac{\ell((v\log v)^{1/2})}{(v\log v)^{3/2}}\\=\frac{\ell((v\log v)^{1/2})}{(\log v)^2}~\to~0,\quad v\to\infty.
\end{multline*}
Here, the last limit relation is secured by $\lim_{v\to\infty}\frac{\ell(v)}{\log v}=0$. We have proved that the contribution of the normal tail dominates, and \eqref{eq:inter3} follows.

Assume now that $\me [\xi^3]<\infty$. According to Theorem 14 on p.~125 in \cite{Petrov:1975}, the assumption $\me [\xi^3]<\infty$ ensures the existence of a universal constant $C>0$ such that, for all $x\in\mr$ and all $n\in\mn$,
$$
\Big|\mmp\{S_n-\mu n\leq xn^{1/2}\}-\Phi(x)\Big|
\leq
\frac{C\me [\xi^3]}{n^{1/2}(1+|x|^3)}.
$$
With the same choice of $n$ and $x$ as in \eqref{eq:choice} we infer
\begin{multline*}
a(v)\Big|\mmp\{S_{\lfloor v+ta(v)\rfloor}>m(v)+ya(v)\}-\Big(1-\Phi\Big(\frac{m(v)-\mu\lfloor v+ta(v)\rfloor+ya(v)}{(\lfloor v+ta(v)\rfloor)^{1/2}}\Big)\Big)\Big|\\=a(v)\Big|\mmp\Big\{S_{\lfloor v+ta(v)\rfloor}-\mu\lfloor v+ta(v)\rfloor\leq \frac{m(v)-\mu \lfloor v+ta(v)\rfloor+ya(v)}{(\lfloor v+ta(v)\rfloor)^{1/2}} (\lfloor v+ta(v)\rfloor)^{1/2}\Big\}\\-\Phi\Big(\frac{m(v)-\mu\lfloor v+ta(v)\rfloor+ya(v)}{(\lfloor v+ta(v)\rfloor)^{1/2}}\Big)\Big|\\\leq \frac{Ca(v)}{(\lfloor v+ta(v)\rfloor)^{1/2}}\Big(1+\frac{m(v)-\mu\lfloor v+ta(v)\rfloor+ya(v)}{(\lfloor v+ta(v)\rfloor)^{1/2}}\Big)^{-1}~\sim~ \frac{C}{(\log v)^2}~\to~0,\quad v\to\infty.
\end{multline*}
This shows that, with $t,y\in\mr$ fixed, as $v\to\infty$, $$a(v)\mmp\{S_{\lfloor v+ta(v)\rfloor}>m(v)+ya(v)\}~\sim~a(v)\Big(1-\Phi\Big(\frac{m(v)-\mu\lfloor v+ta(v)\rfloor+ya(v)}{(\lfloor v+ta(v)\rfloor)^{1/2}}\Big)~\to~\eee^{\mu t-y}.$$ Here, the last limit relation follows from \eqref{eq:normal}.

To prove the limit relation stated in Proposition \ref{prop:vague5} it is enough to check that
\begin{multline*}
\lim_{v\to\infty} \me\Big[\exp\Big(-\sum_{k\geq 1}f\big((k-v)/a(v), (\hat S_k-m(v))/a(v)\big)\Big)\Big]\\=\exp\Big(-\int_{\mr \times \mr} (1-\eee^{-f(t,y)})\eee^{\mu t-y}{\rm d}t\, {\rm d}y\Big)
\end{multline*}
for any nonnegative continuous function $f$ on $\mr\times \mr$ with compact support in $\mr\times (-\infty,\infty]$. In particular, for fixed $f$, there exist $\lambda>0$ and $\gamma\in\mr$ such that $f(t,y)=0$ whenever either $|t|>\lambda$ or $y<\gamma$.

The proof of the limit relation is analogous to the proof of Proposition \ref{prop:vague}. A slightly different argument is only needed for
\begin{equation*}
\lim_{v\to\infty}\sup_{k\geq 1}\int_{(\gamma,\infty)}\big(1-\eee^{-f((k-v)/a(v),y)}\big)\mmp\{(S_k-m(v))/a(v)\in {\rm d}y\}=0.
\end{equation*}
This is a consequence of
\begin{multline*}
\max_{v-\lambda a(v)\leq k\leq v+\lambda a(v)}\,\int_{(\gamma,\infty)}\big(1-\eee^{-f((k-v)/a(v),y)}\big)\mmp\{(S_k-m(v))/a(v)\in {\rm d}y\}\\\leq \max_{1 \leq k\leq v+\lambda a(v)}\,\mmp\{\big(S_k-m(v)\big)/a(v)>\gamma\}\leq
\mmp\{S_{\lfloor v+\lambda a(v)\rfloor}>m(v)+\gamma a(v)\}~\sim~ \eee^{\mu \lambda-\gamma}/a(v)~\to~0
\end{multline*}
as $v\to\infty$. Here, the penultimate limit relation is secured by the previous part of the proof.
\end{proof}

\begin{proof}[Proof of Theorem \ref{thm:05}]
Write $$\frac{\max_{1\leq k\leq \lfloor v+ta(v)\rfloor}\,\hat S_k-m(v)}{a(v)}=\max_{k\in\mn: (k-v)/a(v)\leq t}\frac{\hat S_k-m(v)}{a(v)}.$$ By the argument given on p.~214 in \cite{Resnick:1987} the functional $T:M_p(\mr\times (-\infty,\infty])\to D(\mr)$ defined by $$\Big(T\sum_{k}\varepsilon_{(\tau_k, y_k)}(t)\Big)_{t\in\mr}:=(\max_{k: \tau_k\leq t}\,y_k)_{t\in\mr}$$ is almost surely continuous at $\mathcal{P}^{(3)}$ in the $J_1$-topology. Hence, relation \eqref{eq:basic05} follows from this observation in combination with Proposition \ref{prop:vague5}.

To find the marginal distributions of the limit process $X_3$, write for $t,y\in\mr$,
\begin{equation*}
\me \big[\mathcal{P}^{(3)}\big((s,x): s\leq t, x>y\big)\big]=\int_{-\infty}^t \eee^{\mu s}{\rm d}s\int_y^\infty \eee^{-x}{\rm d}x=\mu^{-1}\eee^{\mu t-y}.
\end{equation*}
Hence,
\begin{multline*}
\mmp\big\{\sup_{k:~t_k \leq t}\,j_k \leq y\big\}=\mmp\big\{\mathcal{P}^{(3)}\big((s,x): s\leq t, x>y\big)=0\big\}\\=\exp\big(-\me \big[\mathcal{P}^{(3)}\big((s,x): s\leq t, x>y\big)\big]\big)=
\exp(-\mu^{-1}\eee^{\mu t-y}).
\end{multline*}
\end{proof}

\subsection{Proof of Theorem \ref{thm:06}}

\begin{assertion}\label{prop:vague6}
Under the assumptions of Theorem \ref{thm:06}, $$\sum_{k\geq 1}\varepsilon_{((k-v)/a(v), (\hat S_k-\mu k)/a(v))}~\Longrightarrow~ \mathcal{P}_3^{(2)},\quad v\to\infty$$ on $M_p(\mr \times ((2/A)^{1/2},\infty])$.
\end{assertion}
\begin{proof}
We prove that, for fixed $t\in\mr$ and $y>0$,
\begin{align}\label{eq:inter4}
\nonumber
&\lim_{v\to\infty}a(v)\mmp\{S_{\lfloor v+ta(v)\rfloor}-\mu \lfloor v+ta(v)\rfloor>ya(v)\}\\
&\hspace{3cm}=
\begin{cases}
y^{-3}, &   \text{if } \ y>(2/A)^{1/2},   \\
(A/(4\pi))^{1/2}+(A/2)^{3/2}, & \text{if} \ y=(2/A)^{1/2}, \\
+\infty, & \text{if} \ y<(2/A)^{1/2}.
\end{cases}
\end{align}
We are going to use formula \eqref{eq:nagaev} with $n=\lfloor v+ta(v)\rfloor$ and $x=ya(v)$. The formula applies whenever the distribution right tail of $\xi$ is regularly varying at $\infty$ of index $-\theta$ for $\theta>2$, see Theorem 1.9 in \cite{Nagaev:1979}.

It can be checked that $\lim_{v\to\infty}va(v)\mmp\{\xi>a(v)\}=1$, whence $$\lim_{v\to\infty}\lfloor v+ta(v)\rfloor a(v)\mmp\{\xi>ya(v)\}=y^{-3},\quad y>0.$$ Further,
\begin{multline*}
a(v)\Big(1-\Phi\Big(\frac{ya(v)}{v^{1/2}}\Big)\Big)~\sim~\frac{a(v)}{(A\pi\log v)^{1/2}y}\eee^{-A (y^2/4)\log v}=\frac{1}{2^{1/2}yv^{Ay^2/4-1/2}}\\~\to~
\begin{cases}
0, &   \text{if } \ y>(2/A)^{1/2},   \\
(A/(4\pi))^{1/2}, & \text{if} \ y=(2/A)^{1/2}, \\
+\infty, & \text{if} \ y<(2/A)^{1/2}.
\end{cases}
\end{multline*}
With these at hand, an application of \eqref{eq:nagaev} yields \eqref{eq:inter4}.

We omit the remaining part of the proof, for it mimics the proof of Proposition \ref{prop:vague3}.
\end{proof}

\begin{proof}[Proof of Theorem \ref{thm:06}]
Put $A_v(t):=\{k\in\mn: k-v\leq ta(v),\ \hat S_k-\mu k\leq (2/A)^{1/2}a(v)\}$ and $B_v(t):=\{k\in\mn: k-v\leq ta(v),\ \hat S_k-\mu k> (2/A)^{1/2}a(v)\}$. We use a representation
\begin{multline*}
\frac{\max_{1\leq k\leq \lfloor v+ta(v)\rfloor}\,\hat S_k-\mu v }{a(v)}\\=\max\Big(\max_{k\in A_v(t)}\,\Big(\frac{\mu(k-v)}{a(v)}+\frac{\hat S_k-\mu k}{a(v)}\Big), \max_{k\in B_v(t)}\,\Big(\frac{\mu(k-v)}{a(v)}+\frac{\hat S_k-\mu k}{a(v)}\Big)\Big)\\=:\max (Y_{1,v}(t), Y_{2,v}(t)).
\end{multline*}
Proposition \ref{prop:vague6} in combination with the Skorokhod representation theorem and Proposition \ref{deterministic1} ensure that $$(Y_{2,v}(t))_{t\in\mr}~\Longrightarrow~(\sup_{k:\,t_k \leq t}\,(\mu t_k+j_k))_{t\in\mr},\quad v\to\infty$$ in the $J_1$-topology on $D(\mr)$. Thus, the claim in Theorem \ref{thm:06} follows if we can show that
\begin{equation}\label{eq:inter5}
(Y_{1,v}(t))_{t\in\mr}~\Longrightarrow~(\mu t+(2/A)^{1/2})_{t\in\mr},\quad v\to\infty
\end{equation}
in the $J_1$-topology on $D(\mr)$.

Since, for each $v>1$, $Y_{1,v}$ is a.s.\ nondecreasing, and the limit process is continuous, it suffices to prove weak convergence of the one-dimensional distributions. Fix any $t\in\mr$. Since $Y_{1,v}(t)\leq \mu t+(2/A)^{1/2}$ a.s., we are left with checking that $\lim_{v\to\infty}\mmp\{Y_{1,v}(t)\leq y\}=0$ for $y<\mu t+(2/A)^{1/2}$. To this end, write
\begin{multline*}
\mmp\{Y_{1,v}(t)\leq y\}=\prod_{k=1}^{\lfloor v+ta(v)\rfloor}\mmp\{S_k\leq \mu v+ya(v),\, S_k\leq \mu k+(2/A)^{1/2}a(v)\}\\=\prod_{v+\mu^{-1}(y-(2/A)^{1/2}) a(v)\leq k\leq v+ta(v)}\mmp\{S_k\leq \mu v+ya(v)\}\\\times \prod_{1\leq k<v+\mu^{-1}(y-(2/A)^{1/2}) a(v)}\mmp\{S_k\leq \mu k+(2/A)^{1/2}a(v)\}\\\leq \prod_{v+\mu^{-1}(y-(2/A)^{1/2}) a(v)\leq k\leq v+ta(v)}\mmp\{S_k\leq \mu v+ya(v)\}\\\leq \exp\Big(-\sum_{v+\mu^{-1}(y-(2/A)^{1/2}) a(v)\leq k\leq v+ta(v)}\mmp\{S_k>\mu v+ya(v)\}\Big).
\end{multline*}
It remains to prove that
\begin{multline*}
\sum_{v+\mu^{-1}(y-(2/A)^{1/2}) a(v)\leq k\leq v+ta(v)}\mmp\{S_k>\mu v+ya(v)\}\\=a(v)\int_{(\lfloor v+\mu^{-1}(y-(2/A)^{1/2})a(v)\rfloor-v)/a(v)}^{(\lfloor v+ta(v)\rfloor-v+1)/a(v)} \mmp\{S_{\lfloor v+xa(v)\rfloor}-\mu\lfloor v+xa(v)\rfloor>\mu v-\mu\lfloor v+xa(v)\rfloor+ya(v)\}{\rm d}x\\~\to~ \infty,\quad v\to\infty.
\end{multline*}
According to \eqref{eq:inter4}, $$\lim_{v\to\infty}a(v)\mmp\{S_{\lfloor v+xa(v)\rfloor}-\mu\lfloor v+xa(v)\rfloor>(y-\mu x)a(v)\}=\infty$$ whenever $y-\mu x<(2/A)^{1/2}$. By Fatou's lemma,
\begin{multline*}
\liminf_{v\to\infty}a(v)\\\times \int_{(\lfloor v+\mu^{-1}(y-(2/A)^{1/2})a(v)\rfloor-v)/a(v)}^{(\lfloor v+ta(v)\rfloor-v+1)/a(v)} \mmp\{S_{\lfloor v+xa(v)\rfloor}-\mu\lfloor v+xa(v)\rfloor>\mu v-\mu\lfloor v+xa(v)\rfloor+ya(v)\}{\rm d}x\\\geq \int_{\mu^{-1}(y-(2/A)^{1/2})}^t \liminf_{v\to\infty}a(v)\mmp\{S_{\lfloor v+xa(v)\rfloor}-\mu\lfloor v+xa(v)\rfloor>(y-\mu x)a(v)\}{\rm d}x=\infty.
\end{multline*}
This completes the proof of \eqref{eq:inter5}.

Passing to the discussion of the marginal distributions of $X_4$ we first observe that $X_4(t)\geq \mu t+(2/A)^{1/2}$ a.s., so that $\mmp\{X_4(t)\leq y\}=0$ for $t\in\mr$ and $y<\mu t+(2/A)^{1/2}$. Since we assume that the distribution functions are right-continuous, it suffices to consider the case $t\in\mr$ and $y> \mu t+(2/A)^{1/2}$. For this range, we may copy a formula from the proof of Theorem \ref{thm:03}: $\me \big[\mathcal{P}_3^{(2)}\big((s,x): s\leq t, \mu s+x>y\big)\big]=(2\mu)^{-1}(y-\mu t)^{-2}$. The claimed formula for $\mmp\{X_4(t)\leq y\}$ in this range is an immediate consequence.
\end{proof}

\subsection{Proof of Theorem \ref{thm:02tau}}

Recall that $a$ varies regularly at $\infty$ of index $2/\alpha$ for $\alpha\in (0,2]$. 
By Theorem 1.8.3 in \cite{Bingham+Goldie+Teugels:1989}, there exists a continuous and strictly increasing function $b$ satisfying $b(v)\sim a(v)$ as $v\to\infty$. As a consequence, $\lim_{v\to\infty}v^2\mmp\{\xi>b(v)\}=1$. Thus, without loss of generality we can and do assume that $a$ is continuous and strictly increasing.

Assume that $\alpha\in (0,2)$ or $\alpha=2$ and $\lim_{v\to\infty}\ell(v)=\infty$. By the Skorokhod representation theorem, after passing to a suitable probability space we may replace the weak convergence in Theorem~\ref{thm:02tau} (respectively, Proposition~\ref{prop:vague}) by almost sure convergence of random elements in the Skorokhod space $D([0,\infty))$ equipped with the $J_1$-topology (respectively, in $M_p([0,\infty)\times(0,\infty])$ equipped with the vague topology). A general result of Whitt~\cite{Whitt:1971} states that the map which assigns to a function $f\in D([0,\infty))$ with $\limsup_{t\to\infty} f(t)=\infty$ its generalized inverse $f^{\leftarrow}$ is continuous in the $M_1$-topology but not in the $J_1$-topology (for a counterexample, see p.~419 in~\cite{Whitt:1971}). Recall also that the $M_1$-topology is weaker than the $J_1$-topology. Consequently, starting from Theorem~\ref{thm:02tau} and passing to the associated first-passage-time processes yields
\begin{multline}
\big(v^{-1}\hat\tau(ta(v))\big)_{t\ge 0}
=\big(\inf\{z\ge 0:\ \hat S_{\lfloor zv\rfloor}>ta(v)\}\big)_{t\ge 0}
\\ \longrightarrow\
\big(\inf\big\{z\ge 0:\ \sup_{k:\,t_k\le z} j_k>t\big\}\big)_{t\ge 0}
=\big(\inf\{t_k:\ j_k>t\}\big)_{t\ge 0},\qquad v\to\infty
\label{eq:in2}
\end{multline}
on $D([0,\infty))$ endowed with the $M_1$-topology. To show that the same convergence holds in the $J_1$-topology, we invoke Lemma~2.4(b) in~\cite{Iksanov+Pilipenko+Samoilenko:2017}. Recall that $X_1(t)=\sup_{k:\,t_k\le t} j_k$ and $X_1^{\leftarrow}(j)=\inf\{t_k:\ j_k>j\}$. Let $j'$ be a discontinuity point of $X_1^{\leftarrow}$; equivalently, $j'$ is a value attained on a constancy interval of $X_1$. More precisely, there exist two atoms $(t',j')$ and $(t'',j'')$ of the point process $\sum_{k\ge1}\delta_{(t_k,j_k)}$ such that $X_1(t)=j'$ for all $t\in[t',t'')$, $X_1(t'-)<j'$, and $X_1(t'')>j'$. By Proposition~\ref{prop:vague},
\[
\sum_{\ell\ge1}\varepsilon_{(\ell/v,\ \hat S_\ell/a(v))}\ \longrightarrow\ \sum_{k\ge1}
\varepsilon_{(t_k,j_k)}
\]
vaguely as $v\to\infty$. Hence there exist atoms $(t'(v),j'(v))$ and $(t''(v),j''(v))$ of $\sum_{\ell\ge1}\varepsilon_{(\ell/v,\ \hat S_\ell/a(v))}$ converging to $(t',j')$ and $(t'',j'')$, respectively, as $v\to\infty$, and for all sufficiently large $v$ the first-passage-time process $\big(\inf\{z\ge0:\ \hat S_{\lfloor zv\rfloor}>ja(v)\}\big)_{j\ge0}$ has a jump at $j=j'(v)$ with left limit $t'(v)$ and right limit $t''(v)$. Since $t'(v)\to t'$ and $t''(v)\to t''$, the conditions of Lemma~2.4(b) in~\cite{Iksanov+Pilipenko+Samoilenko:2017} are satisfied, and the $M_1$-convergence in~\eqref{eq:in2} can be lifted to $J_1$-convergence.


Let $a^{-1}$ be the inverse function of $a$. Observe that $a$ is asymptotic generalized inverse of $x\mapsto 1/(\mmp\{\xi>x\})^{1/2}$, which particularly implies that $a^{-1}(v)\sim 1/(\mmp\{\xi>v\})^{1/2}$ as $v\to\infty$. Substituting now $a^{-1}(v)$ in place of $v$ on the left-hand side of \eqref{eq:in2} we arrive at \eqref{eq:inverseextremal}. The argument in the case $\alpha=2$ and $\liminf_{v\to\infty}\ell(v)<\infty$ is analogous.

For $s,u\geq 0$, put $F(s, u):=\mmp\{X_1(s)\leq u\}=\exp(-s^2 u^{-\alpha}/2)$, where the last equality follows from Theorem \ref{thm:02}. Since $X^\leftarrow_1$ is a nonnegative process, we conclude that $\mmp\{X^\leftarrow_1(t)\leq y\}=0$ for $y<0$. For $y\geq 0$, formula \eqref{eq:marginaltau} is a consequence of $\mmp\{X^\leftarrow_1(t)\leq y\}=1-F(y,t)$.

\subsection{Proofs of Theorems \ref{thm:03tau} and \ref{thm:06tau} 
}

We only give a complete proof of Theorem \ref{thm:03tau}. Theorem \ref{thm:06tau} follows along similar lines.  

Using Proposition \ref{prop:vague3} in combination with the result of Whitt~\cite{Whitt:1971} mentioned above,  
we conclude with the help of the continuous mapping theorem that
\begin{multline}
\Big(\frac{\hat \tau(\mu v+ta(v))-v}{a(v)}\Big)_{t\in\mr 
}=\big(\inf\{z\in\mr: \max_{1\leq k\leq \lfloor v+za(v)\rfloor}\,\hat S_k>\mu v+ta(v)\}\big)_{t\in\mr}\\~\Longrightarrow~ \big(\inf\big\{z\in\mr: \sup_{k:~t_k \leq z}\, j_k >t\big\}\big)_{t\in\mr},\quad v\to\infty \label{eq:in3}
\end{multline}
in the $M_1$-topology on $D(R)$. Arguing as in the proof of Theorem \ref{thm:02tau}, this $M_1$-convergence can be lifted to $J_1$-convergence.
Substituting now $\mu^{-1}v$ in place of $v$ on the left-hand side of \eqref{eq:in3} and recalling that the function $a$ is regularly varying at $\infty$ of index $1/(\alpha-1)$ we obtain \eqref{eq:inverseextremal3}.

Fix any $t\in\mr$. Since $X_2(t)>\mu t$ a.s., we infer $X_2^\leftarrow(t)<\mu^{-1}t$ a.s. Hence, $\mmp\{X^\leftarrow_2(t)\leq y\}=1$ for $y\geq \mu^{-1}t$. For $y<\mu^{-1}t$, formula \eqref{eq:marginaltau03} follows from \eqref{eq:marginaltau} and $\mmp\{X^\leftarrow_2(t)\leq y\}=\mmp\{X_2(y)>t\}$.

\subsection{Proof of Theorem \ref{thm:05tau}}

The functional $Q: M_p(\mr\times (-\infty,\infty])\to D(\mr)$ defined by $$\Big(Q\sum_{k}\varepsilon_{(\tau_k, y_k)}(t)\Big)_{t\in\mr}:=\big(\inf\{z\in\mr: \sup_{k:\,\tau_k \leq z}\,y_k> t\}\big)_{t\in\mr}=
(\inf\{\tau_k: y_k>t\})_{t\in\mr}$$ is almost surely continuous at $\mathcal{P}^{(3)}$ in the $J_1$-topology, see the proof of Corollary 4.21 on p.~216 in \cite{Resnick:1987}. Using Proposition \ref{prop:vague5} in combination with this observation we conclude with the help of the continuous mapping theorem that
\begin{multline*}
\Big(\frac{\hat \tau(m(v)+ta(v))-v}{a(v)}\Big)_{t\in\mr}=\big(\inf\{z\in\mr: \max_{1\leq k\leq \lfloor v+za(v)\rfloor}\,\hat S_k>m(v)+ta(v)\}\big)_{t\in\mr}\\~\Longrightarrow~ \big(\inf\big\{z\in\mr: \sup_{k:~t_k \leq z}\, j_k >t\big\}\big)_{t\in\mr}=\big(\inf\{t_k: ~ j_k >t\}\big)_{t\in\mr},\quad v\to\infty
\end{multline*}
in the $J_1$-topology on $D(R)$.

The formula for the marginal distributions of $X_3^\leftarrow$ follows as in the proof of Theorem \ref{thm:03tau}.

\section{Appendix}

We give a deterministic continuity result which is essentially used when passing from the vague convergence of random point measures to the weak convergence of maxima on the Skorokhod space. The result is a slight extension of Theorem 1.3.17 in \cite{Iksanov:2016} which covers point measures on both $[0,\infty) \times (-\infty, \infty]$ and $\mr\times (-\infty, \infty]$ rather than just on $[0,\infty)\times (-\infty, \infty]$.

Let $I$ denote either $\mr$ or $[0,\infty)$ and $M^\ast_p$ denote the set of 
point measures $\rho$ on $I \times (-\infty,\infty]$ which satisfy  
\begin{equation}\label{1_0_1}
\rho(I_T \times\left((-\infty,-\delta]\cup[\delta,\infty]\right))<\infty
\end{equation}
for all $\delta>0$ and all $T>0$, where $I_T=[0,T]$ if $I=[0,\infty)$ and $I_T=[-T, T]$ if $I=\mr$. The set $M^\ast_p$ is endowed with
the vague topology. For a continuous function $f:I\to\mr$, define the mapping $\mathcal{F}$ from $M^\ast_p$ to $D(I)$ by
\begin{equation*}
\mathcal{F}\left(\rho\right)(t):=
\begin{cases}
        \underset{k:\ \theta_k\leq t}{\sup}(f(\theta_k)+y_k),  & \text{if} \ \theta_k\leq t \ \text{for
some} \ k,\\
        f(0), & \text{otherwise},
\end{cases}
\end{equation*}
where $\rho = \sum_k \varepsilon_{(\theta_k,\,y_k)}$. Assumption
\eqref{1_0_1} ensures that $\mathcal{F}(\rho)\in D(I)$. If \eqref{1_0_1}
does not hold, $\mathcal{F}(\rho)$ may lose right-continuity.

\begin{assertion}\label{deterministic1}
For $n\in\mn$, let $\rho_n\in M_p$. Assume that
\begin{itemize}

\item $\rho_0(I\times(-\infty,0])=0$ and, if $I=[0,\infty)$, that $\rho_0(\{0\}\times (-\infty,+\infty])=0$,

\item $\rho_0((r_1,r_2)\times(0,\infty])\geq 1$ for all $r_1, r_2\in I$ such that $r_1<r_2$,

\item $\rho_0 = \sum_k \varepsilon_{\big(\theta^{(0)}_k,\,y^{(0)}_k\big)}$ does not have clustered jumps, that is, $\theta^{(0)}_k\neq \theta^{(0)}_j$ for $k\neq j$.
\end{itemize}

\noindent If
\begin{equation*}
\lim_{n\to\infty} \rho_n|_{I\times(0,\infty]}=\rho_0
\end{equation*}
on $M^\ast_p$, then
\begin{equation*}
\lim_{n\to\infty} \mathcal{F}(\rho_n)= \mathcal{F}(\rho_0)
\end{equation*}
in the $J_1$-topology on $D(I)$.
\end{assertion}

The proof is omitted, for it only requires obvious modifications of the proof of Theorem 1.3.17 in \cite{Iksanov:2016}.

\noindent \textbf{Acknowledgment.} A part of this work was done while A.I. was visiting the University of Bielefeld and the University of M\"{u}nster in June 2025. Grateful acknowledgment is made for financial support and hospitality. Z.K. has been supported by the German Research Foundation under Germany's Excellence Strategy  EXC 2044/2 -- 390685587, Mathematics M\"unster: Dynamics - Geometry - Structure and by the DFG priority program SPP 2265 Random Geometric Systems. The work V.W. has been funded by the Deutsche Forschungsgemeinschaft (DFG, German Research Foundation) –
Project-ID 317210226 – SFB 1283

%
%


\begin{thebibliography}{30}

\bibitem{Adhikari:2018} K. Adhikari, \textit{Hole probabilities for $\beta$-ensembles and determinantal point processes in the complex plane}. Electron. J. Probab. \textbf{23} (2018), article no. 48, 21 pp.

\bibitem{Akemann+Strahov:2013} G. Akemann and E. Strahov, \textit{Hole probabilities and overcrowding estimates for products of complex Gaussian matrices}. J. Stat. Phys. \textbf{151} (2013), 987--1003.

\bibitem{Alsmeyer+Iksanov+Kabluchko:2025} G. Alsmeyer, A. Iksanov and Z. Kabluchko, \textit{On decoupled standard random walks}. J. Theoret. Probab. \textbf{38} (2025), paper no. 23, 34 pp.

\bibitem{Hough+Krishnapur+Peres+Virag:2009} J. Ben Hough, M. Krishnapur, Y. Peres and B. Vir\`{a}g, \textit{Zeros of Gaussian analytic functions and determinantal point processes}. American Mathematical Society, 
    2009.

\bibitem{Berger2019} Q. Berger, \textit{Notes on random walks in the Cauchy domain of attraction}. Probab. Theory and Relat. Fields. \textbf{175} (2019), 1--44.

\bibitem{Billingsley:1968} P.~Billingsley, \textit{Convergence of probability measures}. Wiley, 1968.

\bibitem{Bingham+Goldie+Teugels:1989} N.H.~Bingham, C.M.~Goldie and J.L.~Teugels, \textit{Regular variation}. Cambridge University Press, 1989.

\bibitem{Buraczewski+Iksanov+Kotelnikova:2025} D. Buraczewski, A. Iksanov and V. Kotelnikova, \textit{Laws of the iterated and single logarithm for sums of independent indicators, with applications to the Ginibre point process and Karlin's occupancy scheme}. Stochastic Process. Appl. \textbf{183} (2025), paper no. 104597, 32 pp.

\bibitem{Byun+Forrester:2025} S.-S. Byun and P.J. Forrester, \textit{Progress on the study of the Ginibre ensembles}. Springer, 2025.


\bibitem{Ethier+Kurtz:2005} S.N. Ethier and T.G. Kurtz, \textit{Markov processes: characterization and convergence}, Wiley, 2005.

\bibitem{Fenzl+Lambert:2022} M. Fenzl and G. Lambert, \textit{Precise deviations for disk counting statistics of invariant determinantal processes}. Int. Math. Res. Not. \textbf{2022} (2022), 7420-–7494.

\bibitem{Iksanov:2016} A. Iksanov, \textit{Renewal theory for perturbed random walks and similar processes}, Birkh\"{a}user, 2016.

\bibitem{Iksanov+Kondratenko:2021} A. Iksanov and O. Kondratenko, \textit{Functional limit theorems for discounted exponential
functional of random walk and discounted convergent perpetuity}. Stat. Probab. Letters. \textbf{176} (2021), 109148.


\bibitem{Iksanov+Pilipenko+Samoilenko:2017} A. Iksanov, A. Pilipenko, and I. Samoilenko, \textit{Functional limit theorems for the maxima of perturbed random walk and divergent perpetuities in the $M_1$-topology}.
Extremes. \textbf{20} (2017), 567--583.


\bibitem{Nagaev:1979} S.V. Nagaev, \textit{Large deviations of sums of independent random variables}. Ann. Probab. \textbf{7} (1979), 745--789.

\bibitem{Petrov:1975} V.V. Petrov, \textit{Sums of independent random variables}, Springer-Verlag, 1975.

\bibitem{Resnick:1987} S. Resnick, \textit{Extreme values, regular variation, and point processes}. Springer-Verlag, 1987.

\bibitem{Whitt:1971} W.~Whitt, \textit{Weak convergence of first passage time processes}. J. Appl. Probab. \textbf{8} (1971), 417--422.


\bibitem{Whitt:2002} W.~Whitt, \textit{Stochastic-process limits: an introduction to stochastic-process limits and their application to queues}. Springer, 2002.




\end{thebibliography}
\end{document}